\documentclass[preprint,12pt]{elsarticle}

\makeatletter
\def\ps@pprintTitle{%
 \let\@oddhead\@empty
 \let\@evenhead\@empty
 \def\@oddfoot{\centerline{\thepage}}%
 \let\@evenfoot\@oddfoot}
\makeatother
\makeatletter
\def\paragraph{\@startsection{paragraph}{4}%
  \z@\z@{-\fontdimen2\font}%
  {\normalfont\bfseries}}
\makeatother

\usepackage[T1]{fontenc} 
\usepackage[english]{babel}
\usepackage{verbatim}
\usepackage{amsmath,amsthm,amssymb}
\usepackage{amsthm}
\usepackage{xcolor}
\usepackage{bm}
\theoremstyle{plain}
\usepackage{graphicx, caption, subcaption}

\usepackage{float}
\usepackage{amsmath, amssymb}
\usepackage{enumitem, latexsym}  
\usepackage{geometry, natbib} 
\usepackage{tabularx}

\usepackage{colortbl}
\usepackage{color}
\definecolor{Gray}{gray}{0.9}
\usepackage{longtable}
\usepackage[acronym]{glossaries}

\usepackage[referable]{threeparttablex}
\renewlist{tablenotes}{enumerate}{1}
\makeatletter
\setlist[tablenotes]{label=\tnote{\alph*},ref=\alph*,itemsep=\z@,topsep=\z@skip,partopsep=\z@skip,parsep=\z@,itemindent=\z@,labelindent=\tabcolsep,labelsep=.2em,leftmargin=*,align=left,before={\footnotesize}}
\makeatother

\usepackage{hhline} 

\usepackage{multirow}
\usepackage{adjustbox}
\usepackage{fancyhdr}
 
\fancypagestyle{plain}{%
   \fancyhf{}
   \fancyfoot[C]{\iffloatpage{}{\thepage}}
   }
\newtheorem{thm}{Theorem}[section]
\newtheorem{cor}[thm]{Corollary}
\newtheorem{lem}[thm]{Lemma}

\newtheorem{prop}[thm]{Proposition}
\theoremstyle{definition}

\theoremstyle{remark}
\newtheorem{remark}[thm]{Remark}

\usepackage{upref}
\usepackage{graphicx, caption, subcaption}
\usepackage{color}
\usepackage{listings}
\usepackage{epstopdf}

\usepackage{mathrsfs}
\usepackage{multirow}
\providecommand{\mathscr}[1]{\mathcal{#1}}

\usepackage{graphicx}
\graphicspath{ {Figures/} }
\usepackage{booktabs}
\usepackage{hyperref}
\usepackage{geometry, natbib} 
\usepackage{adjustbox}
\usepackage[autostyle]{csquotes} 
\usepackage[normalem]{ulem}

\def\R{{\Bbb R}}

\def\M{{\cal M}}

\def\N{{\Bbb N}}

\def\1{{\bf 1}}

\def\E{{\Bbb E}}

\newcommand{\cF}{\mathcal{F}}
\newcommand{\cP}{\mathcal{P}}
\newcommand{\cB}{\mathcal{B}}
\newcommand{\cW}{\mathcal{W}}
\newcommand{\bP}{\mathbb{P}}

\usepackage[toc,page]{appendix}

\def\correspondingauthor{\footnote{Corresponding author.}}
\newcommand{\UniBA}{%
   Universit\`{a} degli Studi di Bari ``Aldo Moro'',
   Department of~Economics and Finance,
   Largo Abbazia S.~Scolastica,
   Bari, I-70124 Italy}

\begin{document}

\title{On the ergodicity of a three-factor CIR model}

\author[ssm]{Giacomo Ascione} 
\address[ssm]{Scuola Superiore Meridionale, Largo S. Marcellino 10, Napoli, 80138, \texttt{g.ascione@ssmeridionale.it}}

\author[uniro]{Michele Bufalo} 
\address[uniro]{Universit\`a degli Studi di Roma "La Sapienza" - Department of Methods and Models for Economics, Territory and Finance, Via del Castro Laurenziano 9, Roma, I-00185, Tel. +39-06-49766903, \texttt{michele.bufalo@uniroma1.it}}

\author[uniba]{Giuseppe Orlando\correspondingauthor{}}
\address[uniba]{\UniBA, Tel. +39 080 5049218, \texttt{giuseppe.orlando@uniba.it}}

\begin{abstract}
This study introduces the $CIR^3$ model, a three-factor model characterized by stochastic and correlated trends and volatilities. The paper focuses on establishing the Wasserstein ergodicity of this model, a task not achievable through conventional means such as the Dobrushin theorem. Instead, alternative mathematical approaches are employed, including considerations of topological aspects of Wasserstein spaces and Kolmogorov equations for measures. Remarkably, the methodology developed here can also be applied to prove the Wasserstein ergodicity of the widely recognized three-factor Chen model.
\end{abstract}


\maketitle

\newpage

\section{Introduction}
The Cox-Ingersoll-Ross process (CIR for short), initially introduced in \cite{cox1985theory}, continues to hold a significant position in mathematical finance (refer to the survey \cite{sundaresan2000continuous}). Since its inception, various adaptations of this process have emerged. For instance, in \cite{mishura2018stochastic}, a CIR model driven by a fractional Brownian motion is examined, with its properties extensively analyzed in \cite{mishura2018fractional, mishura2018fractional2, mishura2023standard}. Alternatively, in \cite{najafi2017bond}, a mixed Brownian motion is utilized. Furthermore, in \cite{leonenko2013fractional}, a fractional CIR model is derived through a time-change mechanism, with its covariance structure explored in \cite{leonenko2013correlation}. A similar approach is adopted in \cite{ascione2021time} within a broader context. Moreover, in \cite{ascione2023foreign}, a four-factor CIR model incorporating jumps, driven by a variance gamma process, is investigated. Lastly, in \cite{zhu2014limit}, a CIR model with Hawkes jumps is studied.

The methodology proposed in \cite{OMB2018, OMB2019a} aimed to fit interest rates across multiple currencies and tenors while preserving the structure of the original CIR model, even in the presence of negative interest rates. In \cite{OMB2019b}, an appropriate partitioning of the data sample was introduced, allowing for the capture of statistically significant time changes in interest rate volatility and consideration of market dynamics jumps. Unlike conventional use for pricing, these studies focused specifically on forecasting expected rates at a given maturity.
Additionally, the single-factor modified CIR model, known as the CIR\# model, introduced in \cite{OB2021, OMB2019c}, emerged as a powerful tool for forecasting future expected interest rates based on observed financial market data. This approach effectively addresses common challenges encountered in interest rate modeling, including simulating regime switching, clustered volatility, and skewed tails. In this paper, we consider a three-factor CIR model, denoted as $CIR^3$ and previously introduced in \cite{Ceci2024}.



In Section \ref{S:Background}, a more detailed exploration of the background, motivations, and main result is provided. However, in the following, we will briefly outline them.
The utilization of CIR processes alongside other stochastic differential equations extends to multi-factor financial models, such as the Heston model. While the classical Heston model is a two-factor model \cite{heston1993closed}, various multi-factor generalizations have been explored, including the Chen model \cite{Chen1996, Chen1996b}.
The Chen model introduces a bond pricing formula under a non-trivial, three-factor model of interest rates \cite{Chen1996}. In this model, the future short rate depends on the current short rate, the short-term mean of the short rate, and the current volatility of the short rate. The model is inspired by the observation that short rates and their volatility are not constant and mean-reverting, achieved through a square root process excluding negative values \cite{Chen1996}.
Our focus lies on a modified version of the Chen model, named the $CIR^3$ model. Here, the state variables are interconnected through correlated noise, and both the stochastic volatility process and the stochastic mean process play roles in the equation of the price process. We investigate the limit and stationary distribution properties of the $CIR^3$ process, employing a strategy that exploits the cascade structure of the system of stochastic differential equations. The main result, Theorem \ref{thm:pweakergoR}, establishes the $p$-weak ergodicity of the $CIR^3$ process under suitable assumptions on the initial data.
Let's underline that the interest in this problem is not only related to its financial aspects but also to its mathematical structure. We aim to generalize our approach to broader systems of stochastic differential equations with a "leader-follower" structure in future work. This includes models where:
\[ 
\begin{cases}
dX_t = F_X(t,X_t) dt + \sigma_X(t,X_t) dW^{(1)}_t \\
dY_t = F_Y(t,X_t,Y_t) dt + \sigma_Y(t,X_t,Y_t) dW^{(2)}_t
\end{cases}
\]
Here, $W_t = (W^{(1)}_t, W^{(2)}_t)$ is a possibly correlated Brownian motion. This effort requires additional regularity properties of the coefficients which are further examined in the remainder of the paper.\\

The organization of this paper is as follows. Section \ref{S:Background} provides the background, motivations and main result (i.e. Theorem \ref{thm:pweakergoR}). In Section \ref{S:Preliminaries} we present some preliminaries on the topology of the $p$-Wasserstein metric spaces and the $CIR^3$ model.
%
The proof of Theorem \ref{thm:pweakergoR} will be articulated throughout a substantial number of auxiliary results, presented in sections \ref{sec:CIRwe}, \ref{subsec:twofactor}, and \ref{sec:finalsec}. Specifically, Section \ref{sec:CIRwe} explores the $p$-weak ergodicity of the $CIR$ process $v_t$. Despite the strong ergodicity of such a process is already known, the classical result given in \cite[Theorem V.54.5]{rogers2000diffusions} does not provide the $p$-weak ergodicity for $p>1$ or the rate of convergence. For this reason we provide full details on the $p$-Wasserstein convergence of the single-factor process. Moving on, Section \ref{subsec:twofactor}, in accordance with the \textit{cascade structure} of the equation, focuses on establishing the existence of the limit distribution of the two-factor model. This part is subdivided in four steps that will be described in details in what follows. Once the limit distribution of the two-factor model has been provided, in Section \ref{sec:finalsec} we finally prove Theorem \ref{thm:pweakergoR}. Section \ref{Sec:Conclusion} provides some concluding remarks.

\section{Background on the CIR process,  motivations and main result} \label{S:Background}

CIR processes are also utilized in combination with other stochastic differential equations to describe multi-factor financial models, such as the Heston model. While the classical Heston model, introduced in \cite{heston1993closed}, is a two-factor model, several multi-factor generalizations have been explored; for instance, see \cite{christoffersen2009shape} and references therein. Among multi-factor financial models, the Chen model, introduced in \cite{Chen1996, Chen1996b}, holds particular significance.
Chen proposed a bond pricing formula "under a non-trivial, three-factor model of interest rates" \cite{Chen1996}, where the future short rate "depends on (1) the current short rate, (2) the short-term mean of the short rate, and (3) the current volatility of the short rate". The model was inspired by the observation that both short rates and their volatility are not constant and mean-reverting as well. This could be achieved by means of a square root process that excludes negative values and allows for mean reversion \cite{Chen1996}.

Summarizing, Chen model relies on the following system of stochastic differential equations
\begin{equation}\label{eq:SDEaux2c}
	\begin{cases}
		dR_t=k(\theta_t-R_t)+\alpha\sqrt{v_tR_t}dW^{(1)}_t\\
		d\theta_t=k_\theta(\zeta-\theta_t)dt+\beta \sqrt{\theta_t}dW^{(2)}_t\\
		dv_t=k_v(\eta-v_t)dt+\gamma \sqrt{v_t}dW^{(3)}_t,
	\end{cases}
\end{equation}
where $\{(W^{(1)}_t,W^{(2)}_t,W^{(3)}_t), \ t \ge 0\}$ is a three-dimensional standard Brownian motion and $k,k_\theta,k_v,\alpha,\beta,\gamma,\zeta,\eta>0$ are constants. Up to our knowledge, proof of global existence and pathwise uniqueness of a strong solution of \eqref{eq:SDEaux2c} has not been provided. It should be noted that in the Chen model the stochastic volatility process $v$ and the stochastic mean process $\theta$ are independent CIR processes which then play a role in the equation of the price process $R$.


%
In this work, we consider a modified Chen model, called the $CIR^3$ model, where the state variables are interconnected through correlated noise and the stochastic mean process $\theta$ is also subject to the stochastic volatility process $v$. Summarizing, we rely on the system of stochastic differential equations
\begin{equation} \label{Our_Model_true_intro}
	\begin{cases}
		dR_t=k(\theta_t-R_t)dt+\alpha \sqrt{|v_tR_t|}dW^{(1)}_t \\
		d\theta_t=k_\theta(\zeta-\theta_t)dt+\alpha\beta\sqrt{|v_t\theta_t|}dW^{(2)}_t \\
		dv_t=k_v(\eta-v_t)dt+\gamma\sqrt{|v_t|}dW^{(3)}_t,
	\end{cases}
\end{equation}
where $\{(W^{(1)}_t,W^{(2)}_t,W^{(3)}_t), \ t \ge 0\}$ is a three-dimensional correlated Brownian motion, whose correlation matrix will be specified later, and $k,k_\theta,k_v,\alpha,\beta,\gamma,\zeta,\eta>0$ are constants. Both the correlation of the noise and the proportionality of the volatility of $R$ and $\theta$ are realistic assumptions. This model has been first investigated in \cite{Ceci2024}, in which global existence of weak solutions and local existence and pathwise uniqueness of strong solutions of \eqref{Our_Model_true_intro} have been proved. Global existence and pathwise uniqueness of strong solutions of \eqref{Our_Model_true_intro} have been shown in \cite{Ascione2024Jan}, where positivity of the involved factors has been also investigated. Furthermore, parameter estimation and forecasting through the $CIR^3$ model have been discussed in \cite{Ascione2023Dec}. 

Here we focus on the existence of a limit and stationary distribution. Such a property, that is suggested by the mean-reverting nature of \eqref{Our_Model_true_intro}, turns out to be an interesting tool to better investigate financial time series. Indeed, on the one hand, the identification of a stationary distribution can help with parameter estimation and calibration of stationary time series that are modelled through such a model. On the other hand, the existence of a limit distribution can be used to describe the long-time behaviour of models described through the $CIR^3$ process.

Despite being clearly drift-dissipative, classical result on dissipative stochastic differential equations cannot be used on the $CIR^3$ system \eqref{Our_Model_true_intro} due to the presence of a non-Lipschitz diffusion term. In the one-dimensional setting, classical results on the total variation distance, such as \cite[Theorem V.54.5]{rogers2000diffusions}, can be used to prove the strong ergodicity of the $CIR$ process. These results clearly do not apply to the $3$-dimensional setting. Hence, to study the limit distribution of the $CIR^3$ process, we rely on the ergodicity in the $p$-Wasserstein distance, which is also called $p$-weak ergodicity. Despite the name, it is interesting to observe that the strong ergodicity (in total variation) only implies the $1$-weak ergodicity, while one needs a weighted total variation distance to imply the case $p>1$, see \cite[Theorem 6.15]{villani2009optimal}. For this reason, despite the name, the $p$-weak ergodicity could be not comparable  with the strong ergodicity. The main advantage of such an approach is the fact that to provide estimates of the $p$-Wasserstein distances we can work directly with the stochastic differential equations by means of It\^{o}'s formula and further moment estimates. This approach is classical in the theory of $p$-weak ergodicity and leads to a quite general result for drift-dissipative stochastic differential equations with Lipschitz diffusion coefficients, as in \cite[Section 4.2.1]{kulik2017ergodic}. These arguments, however, do not apply to \eqref{Our_Model_true_intro}, due to the lack of Lipschitz property in the diffusion term. For this reason, we develop a strategy that exploits the \textit{cascade} structure of the system \eqref{Our_Model_true_intro} to prove the existence of a stationary distribution and the convergence in $p$-Wasserstein distance of the $CIR^3$ process, under suitable assumptions on the initial data. More precisely, if we denote by $\{\bm{R}_t^{\bm{X}}\}_{t \ge 0}$ the pathwise unique global strong solution of \eqref{Our_Model_true_intro} with initial data $\bm{X}$ on a probability space $(\Omega,\cF,\{\cF_t\}_{t \ge 0}, \bP)$, the main result of the paper is given by the following theorem.
\begin{thm}\label{thm:pweakergoR}
	Let $p \ge 2$, $\overline{p}>2p-1$, $p'>2\overline{p}-1$ and $\widetilde{p}=\max\left\{2\overline{p}-1,\frac{p'}{p'-2\overline{p}+1}\right\}$. Consider three $\cF_0$-random variables $X^{(1)},X^{(2)} \in L^{p'}(\Omega;\R_{+,0}^2)$ and $X^{(3)} \in L^{\widetilde{p}}(\Omega;\R_+)$ and set $\bm{X}=(X^{(1)},X^{(2)},X^{(3)})$. Then the process $\{\bm{R}_t^{\bm{X}}\}_{t \ge 0}$ is $p$-weakly ergodic.
	%
	%
\end{thm}

Let us underline that the interest on such a problem is not only related to the financial aspects of this model, but also to its mathematical structure. Indeed, the strategy we present here clearly applies also to Chen's model and we think it is related to a much more general structure. More precisely, we aim to generalize in future work the $p$-weak ergodicity result to more general systems of stochastic differential equations with a \textit{leader-follower} structure, i.e.
\begin{equation*}
	\begin{cases}
		dX_t=F_X(t,X_t)dt+\sigma_X(t,X_t)dW^{(1)}_t\\
		dY_t=F_Y(t,X_t,Y_t)dt+\sigma_X(t,X_t,Y_t)dW^{(2)}_t,
	\end{cases}
\end{equation*}
where $W_t=(W^{(1)}_t,W^{(2)}_t)$ is a $d_1+d_2$-dimensional, possibly correlated, Brownian motion, $F_X,F_Y,\sigma_X,\sigma_Y$ are sufficiently regular to guarantee global existence and pathwise uniqueness of strong solutions and satisfy a drift-dissipation condition, i.e.
\begin{align}\label{eq:driftdiss}
	\begin{split}
(F_X(t,x_1)-F_X(t,x_2))(x_1-x_2)&<-C|x_1-x_2|^2	\\
(F_Y(t,x,y_1)-F_Y(t,x,y_2))(y_1-y_2)&<-C|y_1-y_2|^2,		
\end{split}
\end{align}
where $C>0$ is a constant. Clearly, since we are not assuming the Lipschitz condition, some additional regularity properties of the coefficients will be required.

\section{Preliminaries and notation} \label{S:Preliminaries}

Before proceeding, let us establish the probability space upon which all the involved processes and random variables will be defined. Given our interest in the distributional properties of certain stochastic processes, the choice of a specific filtered probability space as a framework will not compromise the generality of the results.

\begin{prop}\label{prop:repspace}
	There exists a complete filtered probability space $(\Omega, \cF, \{\cF_t\}_{t \ge 0}, \bP)$ satisfying the usual hypotheses with the following properties:
	\begin{itemize}
		\item[(I)] It supports a $\cF_t$-adapted correlated Brownian motion $W=\{(W^{(1)}_t,W^{(2)}_t,W^{(3)}_t), t \ge 0\}$ whose correlation matrix is given by \eqref{eq:corrmat}.
		\item[$(ii)$] For any $\mu \in \cP(\R)$ there exists a $\cF_0$-measurable random variable $X:\Omega \to \R$ with ${\rm Law}(X)=\mu$.
		\item[(III)] Let $d_1,d_2 \ge 1$, $\mu \in \cP_p(\R^{d_1})$, $\nu \in \cP_p(\R^{d_2})$ and $\pi \in \Pi(\mu,\nu)$. Assume that $X$ is a $\cF_0$-measurable $\R^{d_1}$-valued random variable with ${\rm Law}(X)=\mu$. Then there exists a $\cF_0$-measurable $\R^{d_2}$-valued random variable $Y$ with ${\rm Law}(Y)=\nu$ and such that ${\rm Law}(X,Y)=\pi$.
	\end{itemize}
\end{prop}
The proof is given in  \ref{app:reprspace}. Throughout the paper, we assume that all the involved processes and random variables are defined on the complete filtered probability space $(\Omega, \cF, \{\cF_t\}_{t \ge 0}, \bP)$ obtained in Proposition~\ref{prop:repspace}.
Given $E \in \cF$, we denote by $\mathbf{1}_E$ the indicator function of the event $E$. For any $T>0$, we denote by $M^2_T(\Omega)$ the space of progressively measurable (with respect to $\{\cF_t\}_{t \ge 0}$) stochastic processes $\{X_t\}_{t \ge 0}$ such that $\int_0^T\E[|X_t|^2]ds<\infty$. We also denote $\R_+:=(0,+\infty)$ and $\R_{+,0}:=[0,+\infty)$. Furthermore, $\cB(\R^d)$ will denote the Borel $\sigma$-algebra over $\R^d$, where $d \in \N$. Given a random variable $X:\Omega \to \R^d$, ${\rm Law}(X)$ will be the probability measure on $(\R^d, \cB(\R^d))$ defined as
\begin{equation*}
{\rm Law}(X)(A)=\bP(X \in A), \ \forall A \in \cB(\R^d).
\end{equation*}

\subsection{The $CIR^3$ model and its generator} \label{S:Model}


\noindent Let us introduce the model under consideration. We define the $CIR^3$ model as the pathwise unique global strong solution $\{\bm{R}_t\}_{t \ge 0}=\{(R_t,\theta_t,v_t)\}_{t \ge 0}$ of the system of SDEs
\begin{equation} \label{Our_Model_true}
	\begin{cases}
		dR_t=k(\theta_t-R_t)dt+\alpha \sqrt{|v_tR_t|}dW^{(1)}_t \\
		d\theta_t=k_\theta(\zeta-\theta_t)dt+\alpha\beta\sqrt{|v_t\theta_t|}dW^{(2)}_t \\
		dv_t=k_v(\eta-v_t)dt+\gamma\sqrt{|v_t|}dW^{(3)}_t,
	\end{cases}
\end{equation}
with given initial data $(R_0,\theta_0,v_0)\in \R_+^3$, where $k,k_\theta,k_v,\alpha,\beta,\gamma,\zeta,\eta>0$ and $\{W_t\}_{t \ge 0}=\{(W^{(1)}_t,W^{(2)}_t,W^{(3)}_t)\}_{t \ge 0}$ is a $3$-dimensional correlated Brownian motion adapted to $\{\cF_t\}_{t \ge 0}$ whose infinitesimal correlation matrix is given by
\begin{equation}\label{eq:corrmat}
	\Sigma=\begin{pmatrix} 1 & \rho_\theta & \rho_v \\
		\rho_\theta & 1 & 0 \\
		\rho_v & 0 & 1 \end{pmatrix}
\end{equation}
and the correlation coefficients satisfy $\overline{\rho}:=1-\rho_\theta^2-\rho_v^2>0$.

\noindent The $CIR^3$ model represents a modified Chen-type model in which the variance of the mean process $\theta_t$ is proportional to the variance of the process itself $R_t$ and the noise are correlated. This is a realistic hypothesis. The existence and pathwise uniqueness of a strong solution for \eqref{Our_Model_true} has been shown in \cite[Theorem $1.1$]{Ascione2024Jan}, together with the Feller property. In the following result, we only specify the generator of the Feller processes $\{\bm{R}_t\}_{t \ge 0}$ and $\{\bm{\theta}_t\}_{t \ge 0}:=\{(\theta_t,v_t)\}_{t \ge 0}$.
\begin{prop}\label{prop:RFeller}
	The strong solution $\{\bm{R}_t\}_{t \ge 0}$ of \eqref{Our_Model_true} is a Feller process with generator $\mathcal{G}_{\bm{R}}$ defined on $C^\infty_c(\R_{+}^3)$ as
 \begin{equation}\label{eq:gen3}
 \mathcal{G}_{\bm{R}}\varphi(R,\theta,v)=\mathbf{b}(R,\theta,v)\nabla \varphi(R,\theta,v)+\frac{1}{2}{\rm Tr}\left(\bm{\sigma}^{\bm{t}}(R,\theta,v)\nabla^2\varphi(R,\theta,v)\bm{\sigma}(R,\theta,v)\right),
\end{equation}
where, for a matrix $\bm{A}$, $\bm{A}^{\bm{t}}$ denotes its transposed, and, for $(R,\theta,v) \in \R_+^3$,
 \begin{equation}\label{eq:drfitdiff}
		\mathbf{b}(R,\theta,v)=\begin{pmatrix} k(\theta-R) \\ k_\theta(\zeta-\theta) \\ k_v(\eta-v) \end{pmatrix} \ \mbox{ and } \
		\bm{\sigma}(R,\theta,v)=\begin{pmatrix} 
			\alpha \sqrt{\overline{\rho}vR} & \rho_\theta \alpha \sqrt{vR} & \rho_v \alpha \sqrt{vR} \\
			0 & \alpha \beta \sqrt{v\theta} & 0 \\
			0 & 0 & \gamma \sqrt{v}
		\end{pmatrix}.
\end{equation}
Furthermore, also $\{\bm{\theta}_t\}_{t \ge 0}$ is a Feller process with generator $\mathcal{G}_{\bm{\theta}}$ defined on $C^\infty_c(\R_{+}^2)$ as
 \begin{equation}\label{eq:gen2}
 \mathcal{G}_{\bm{\theta}}\varphi(\theta,v)=\mathbf{b}_{\bm{\theta}}(\theta,v)\nabla \varphi(\theta,v)+\frac{1}{2}{\rm Tr}\left(\bm{\sigma}_{\bm{\theta}}^{\bm{t}}(\theta,v)\nabla^2\varphi(\theta,v)\bm{\sigma}_{\bm{\theta}}(\theta,v)\right),
\end{equation}
where, for $(\theta,v) \in \R_+^2$,
 \begin{equation}\label{eq:drfitdiff}
		\mathbf{b}_{\bm{\theta}}(\theta,v)=\begin{pmatrix} k_\theta(\zeta-\theta) \\ k_v(\eta-v) \end{pmatrix} \ \mbox{ and } \
		\bm{\sigma}_{\bm{\theta}}(\theta,v)=\begin{pmatrix} 			
			\alpha \beta \sqrt{v\theta} & 0 \\
			     0 & \gamma \sqrt{v}
		\end{pmatrix}.
\end{equation}
\end{prop}
\begin{proof}
The fact that $\{\bm{R}_t\}_{t \ge 0}$ is a Feller process has been already proven in \cite[Theorem 1.1]{Ascione2024Jan}. To prove \eqref{eq:gen3}, it is sufficient to observe that we have $W^{(1)}_t=\sqrt{\overline{\rho}}\widetilde{W}^{(1)}_t+\rho_\theta W^{(2)}_t+\rho_vW^{(3)}_t$, where $(\widetilde{W}^{(1)}_t,W^{(2)}_t,W^{(3)}_t)$ is a standard Brownian motion, while \eqref{eq:gen2} follows once we observe that $(W^{(2)}_t,W^{(3)}_t)$ is already a standard two-dimensional Brownian motion.
\end{proof}
Let us stress that \eqref{Our_Model_true} satisfy the drift-dissipativity condition \eqref{eq:driftdiss}, in the sense that setting $b_R(\theta,R)=k(\theta-R)$, it holds
\begin{align*}
	(\bm{b}_{\bm{\theta}}(\theta_1,v_1)-\bm{b}_{\bm{\theta}}(\theta_2,v_2))(\theta_1-\theta_2,v_1-v_2)& \le -\min\{k_v,k_\theta\}|(\theta_1-\theta_2,v_1-v_2)|^2\\
	(b_R(\theta,R_1)-b_R(\theta,R_2))(R_1-R_2)&\le k(R_1-R_2)^2.
\end{align*}
The classical drift-dissipativity condition, as in \cite[Section 4.2.1]{kulik2017ergodic}, is also satisfied if $k<4k_\theta$. Indeed, in such a case, we can set $\varepsilon=\frac{1}{2}\left(1+\sqrt{\frac{k}{2k_\theta}}\right)$ and get, by Young's inequality
\begin{align*}
	(\bm{b}(R_1,&\theta_1,v_1)-\bm{b}(R_2,\theta_2,v_2))(R_1-R_2,\theta_1-\theta_2,v_1-v_2)\\
	&=-k_v(v_1-v_2)^2-k_\theta(\theta_1-\theta_2)^2-k(R_1-R_2)^2+k(\theta_1-\theta_2)(R_1-R_2)\\
	&\le -k_v(v_1-v_2)^2-\left(k_\theta-\frac{k}{2\varepsilon^2}\right)(\theta_1-\theta_2)^2-k\left(1-\frac{\varepsilon^2}{2}\right)(R_1-R_2)^2\\
	&\le -\min\left\{k_v,k_\theta-\frac{k}{2\varepsilon^2},k\left(1-\frac{\varepsilon^2}{2}\right)\right\}|(R_1-R_2,\theta_1-\theta_2,v_1-v_2)|^2.
\end{align*}
This, however, is not enough to argue as in \cite[Section 4.2.1]{kulik2017ergodic} since $\bm{\sigma}$ is not Lipschitz-continuous.
\subsection{$p$-Wasserstein distances and $p$-weak ergodicity}

Now let us introduce some notation on Wasserstein distances.

For any $d \ge 1$, let $\cP(\R^d)$ be the set of all probability measures on $\R^d$. For any two measures $\mu \in \cP(\R^{d_1})$ and $\nu \in \cP(\R^{d_2})$, we say that $\pi \in \cP(\R^{d_1+d_2})$ is a transport plan of $\mu$ and $\nu$ if for any $A_j \in \cB(\R^{d_j})$ with $j=1,2$, $\pi(A_1 \times \R^{d_2})=\mu(A)$ and $\pi(\R^{d_1} \times A_2)=\nu(A_2)$, i.e. if $\pi$ admits $\mu$ and $\nu$ as marginals. The set of all transport plans of $\mu$ and $\nu$ will be denoted as $\Pi(\mu,\nu)$. A coupling of $\mu$ and $\nu$ is a $\R^{d_1+d_2}$-valued random variable $(X,Y)$ such that ${\rm Law}(X)=\mu$ and ${\rm Law}(Y)=\nu$. The set of couplings of $\mu$ and $\nu$ is denoted by $C(\mu,\nu)$. Clearly, any coupling $(X,Y)\in C(\mu,\nu)$ defines a transport plan $\pi={\rm Law}(X,Y)$, while, by Skorokhod representation theorem, any transport plan $\pi \in \Pi(\mu,\nu)$ admits at least a coupling. For $p \ge 1$ and $\mu \in \cP(\R^d)$, we denote
\begin{equation*}
	M_p(\mu)=\int_{\R^d}|x|^pd\mu(x)
\end{equation*}
and we define
\begin{equation*}
\cP_p(\R^d):=\{\mu \in \cP(\R^d): \ M_p(\mu)<\infty\}.
\end{equation*}
On $\cP_p(\R^d)$ we can define the following metric:
\begin{equation*}
	\cW_p(\mu,\nu):=\inf_{\pi \in \Pi(\mu,\nu)}\left(\int_{\R^{2d}}|x-y|^pd\pi(x,y)\right)^{\frac{1}{p}}=\inf_{(X,Y) \in C(\mu,\nu)}\left(\E[|X-Y|^p]\right)^{\frac{1}{p}},
\end{equation*}
that is called the $p$-Wasserstein distance on $\cP_p(\R^d)$. The metric space $(\cP_p(\R^{2d}),\cW_p)$ is called the $p$-Wasserstein space and it is a complete separable metric space (see \cite[Theorem 6.18]{villani2009optimal}). Now let $\{\mu_k\}_{k \ge 0} \subseteq \cP(\R^d)$ and $\mu \in \cP(\R^d)$. We say that $\mu_k$ converges towards $\mu$ weakly in $\cP(\R^d)$ if for any $\varphi \in C_b(\R^d)$, where we denote by $C_b(\R^d)$ the space of bounded continuous functions, we have
\begin{equation*}
\lim_{k \to +\infty}\int_{\R^d}\varphi(x)d\mu_k(x)=\int_{\R^d}\varphi(x)d\mu(x).
\end{equation*}
This will be denoted by $\mu_k \overset{\cP}{\to}\mu$. If $\{\mu_k\}_{k \ge 0} \subseteq \cP_p(\R^d)$ and $\mu \in \cP_p(\R^d)$, we say that $\mu_k$ converges towards $\mu$ weakly in $\cP_p(\R^d)$ if and only if $\mu_k \overset{\cP}{\to} \mu$ and $\lim_{k \to +\infty}M_p(\mu_k)=M_p(\mu)$. Such convergence will be denoted as $\mu_k \overset{\cP_p}{\to} \mu$. Let $\mathcal{M}\subset \cP(\R^d)$ be a family of probability measures. We say that $\mathcal{M}$ is \textit{equitight} if for any $\varepsilon>0$ there exists a compact set $K_\varepsilon \subset \R^d$ such that
\begin{equation*}
	\mu(\R^d \setminus K_\varepsilon)<\varepsilon \qquad \forall \mu \in \mathcal{M}.
\end{equation*}
We say that $\mathcal{M} \subset \cP_p(\R^d)$ has \textit{uniformly integrable $p$-moments} if
\begin{equation*}
	\lim_{R \to +\infty}\sup_{\mu \in \mathcal{M}}\int_{|x| \ge R}|x|^pd\mu(x)=0.
\end{equation*}
We will make use of the following result (see \cite[Definition 6.8 and Theorem 6.9]{villani2009optimal}).
\begin{thm}\label{thm:equivalentWass}
	Let $p \ge 1$, $\{\mu_k\}_{k \ge 0}\subseteq \cP_p(\R^d)$ and $\mu \in \cP(\R^d)$. The following properties are equivalent:
	\begin{itemize}
		\item[(i)] $\mu_k \overset{\cP_p}{\to} \mu$;
		\item[(ii)] $\mu_k \overset{\cP}{\to} \mu$ and
		\begin{equation}\label{eq:uniint}
			\lim_{R \to +\infty}\limsup_{k \to +\infty}\int_{|x| \ge R}|x|^pd\mu_k(x)=0;
		\end{equation}
		\item[(iii)] $\lim_{k \to +\infty}\cW_p(\mu_k,\mu)=0$. 
	\end{itemize}
\end{thm}
\begin{remark}\label{rmk:suffcond}
	Observe that \eqref{eq:uniint} is equivalent to the request that $\{\mu_k\}_{k \ge 0}$ has uniformly integrable $p$-moments. It is not difficult to check that (ii) is verified if $\mu_k \overset{\cP}{\to} \mu$ and there exists $\varepsilon>0$ such that $\sup_{k \ge 0}M_{p+\varepsilon}(\mu_k)<\infty$. Furthermore, recall that if (i), (ii) or (iii) are verified, then $\mu \in \cP_p(\R)$.
\end{remark}
We will also make use of the following characterization of relatively compact subsets of $\cP_p(\R^d)$ (see \cite[Proposition 7.1.5]{ambrosio2005gradient}).
\begin{thm}\label{thm:relcomp}
	Let $p \ge 1$ and $\mathcal{M} \subseteq \cP_p(\R^d)$. $\mathcal{M}$ is relatively compact (with respect to the metric $\cW_p$) if and only if it is equitight and has uniformly integrable $p$-moments.
\end{thm}
\begin{remark}\label{rmk:boundcomp}
    It is not difficult to check that if $\mathcal{M} \subset \cP_p(\R^d)$ is such that $\sup_{\mu \in \mathcal{M}}M_{p+\varepsilon}(\mu)<\infty$ for some $\varepsilon>0$, then $\mathcal{M}$ is relatively compact with respect to the metric $\cW_p$. Indeed, by Remark \ref{rmk:suffcond}, we know that $\mathcal{M}$ has uniformly integrable $p$-moments. Furthermore, if we consider the ball $B_\delta$ centered in the origin with radius $\delta>0$, we get, by Markov's inequality
    \begin{equation*}
    \sup_{\mu \in \mathcal{M}}\mu(\R^d \setminus B_\delta) \le \frac{\sup_{\mu \in \mathcal{\M}}M_{p+\varepsilon}(\mu)}{\delta^{p+\varepsilon}}
    \end{equation*}
    that goes to $0$ as $\delta \to +\infty$. Hence, $\mathcal{M}$ is equitight and we can apply \eqref{thm:relcomp}.
\end{remark}
For further notions on Wasserstein distances, we refer to \cite{ambrosio2005gradient,villani2009optimal}.
 
Let $p \ge 1$ and $\{X_t\}_{t \ge 0}$ be a stochastic process with $\E[|X_t|^p]<\infty$ for any $t \ge 0$. We say that $\{X_t\}_{t \ge 0}$ is weakly $p$-ergodic if there exists a probability measure $\mu_\infty \in \cP_p(\R)$ such that, setting $\mu_t={\rm Law}(X_t)$, it holds $\mu_t \overset{\cP_p}{\to}\mu_\infty$. It is clear that, as a consequence, $\mu_\infty$ is the limit distribution of $X_t$.

The next sections will be devoted to the $p$-weak ergodicity of the process $\{\bm{R}\}_{t \ge 0}$. However, the proof of Theorem~\ref{thm:pweakergoR} first requires a detailed study of the processes $\{v_t\}_{t \ge 0}$ and $\{\bm{\theta}_t\}_{t \ge 0}$, which will be carried on in what follows.
\section{The $p$-weak ergodicity of the $CIR$ process $v_t$} \label{sec:CIRwe}
Let us first focus on the process $\{v_t\}_{t \ge 0}$, that is solution to the classical CIR equation
\begin{equation}\label{eq:v}
	dv_t=k_v(\eta-v_t)dt+\gamma \sqrt{v_t}dW^{(3)}_t.
\end{equation}
The process belongs to the class of Pearson diffusions (see \cite{forman2008pearson}) and it is known that it admits a Gamma limit distribution, whose density is given by
\begin{equation*}
	m_v(x)=\left(\frac{2k_v}{\gamma^2}\right)^{\frac{2k_v\eta}{\gamma^2}}\frac{x^{\frac{2k_v\eta}{\gamma^2}-1}e^{-\frac{2k_v}{\gamma^2}x}}{\Gamma\left(\frac{2k_v\eta}{\gamma^2}\right)}, \ x>0.
\end{equation*}
Let us denote by $\mathbf{m}_v:\cB(\R) \to [0,1]$ the measure $\mathbf{m}_v(A)=\int_{A}m_v(x)dx$. In particular, if ${\rm Law}(v_0)=\mathbf{m}_v$, then, for any $t \ge 0$, ${\rm Law}(v_t)=\mathbf{m}_v$. Let us also recall that $\mathbf{m}_v$ admits finite moments of any order given, for $p \ge 1$, by
\begin{equation}\label{eq:momGamma}
M_p(\mathbf{m}_v):=\int_0^{+\infty}x^pm_v(x)dx=\frac{\Gamma\left(\frac{2k_v\eta}{\gamma^2}+p\right)}{\Gamma\left(\frac{2k_v\eta}{\gamma^2}\right)}\left(\frac{\gamma^2}{2k_v}\right)^p.
\end{equation}
The convergence in distribution of $v_t$ towards a Gamma random variable can be verified using classical spectral decomposition methods. Here we would like to discuss a different approach.

For any $\cF_0$-measurable random variable $X \in L^2(\Omega;\R_+)$, we denote by $\{v_t^{X}\}_{t \ge 0}$ the unique strong solution of \eqref{eq:v} with $v_0^{X}=X$. If $X=x$ a.s. for some $x>0$, we denote it as $v_t^x$. Let us also recall that the expectation of $v_t^X$ is given by
\begin{equation}\label{eq:expecvt}
	\E[v_t^X]=\eta+(\E[X]-\eta)e^{-k_vt}.
\end{equation}
For any $\cF_0$-measurable random variables $X,Y \in L^2(\Omega;\R_+)$ and $t \ge 0$, let $\Delta_t^{v,X,Y}=v_t^{X}-v_t^{Y}$. First, we estimate the first absolute moment of $\Delta_t^{v,X,Y}$.
\begin{lem}\label{lem:boundcondexp}
	Let $X,Y \in L^2(\Omega;\R_{+,0})$ be $\cF_0$-measurable and consider any $t>0$ and $t_0 \in [0,t)$. Then
	\begin{equation}\label{eq:condexp1}
		\E\left[\left.|\Delta_t^{v,X,Y}|\, \right| \Delta_{t_0}^{v,X,Y}\right] \le |\Delta_{t_0}^{v,X,Y}|e^{-k_v(t-t_0)}.
	\end{equation}
\end{lem}
\begin{proof}
	By \eqref{eq:v}, we have
	\begin{align*}
		\Delta_t^{v,X,Y}&=\Delta_{t_0}^{v,X,Y}-k_v\int_{t_0}^t \Delta_{s}^{v,X,Y}ds+\gamma\int_{t_0}^t(\sqrt{v_s^{X}}-\sqrt{v_s^{Y}})dW_t^{(3)}.
	\end{align*}
	Observe that
	\begin{multline*}
		\E[(\sqrt{v_s^{X}}-\sqrt{v_s^{Y}})^2] \le 2(\E[v_s^{X}]+\E[v_s^{Y}]) \\ \le 4\eta+2(\E[X]+\E[Y]-2\eta)e^{-k_vt} \le 8\eta+2(\E[X]+\E[Y])<\infty.
	\end{multline*}
	hence, for any $t \ge 0$, $\{\sqrt{v_s^{X}}-\sqrt{v_s^{Y}}\}_{s \ge 0} \in M^2_t(\Omega)$. Taking the conditional expectation we get
	\begin{align*}
		\E[\Delta_t^{v,X,Y} \mid \cF_{t_0}]&=\Delta_{t_0}^{v,X,Y}-k_v\int_{t_0}^t \E[\Delta_{s}^{v,X,Y} \mid \cF_{t_0}]ds.
	\end{align*}
	Next, we take the conditional expectation with respect to $v_{t_0}^X,v_{t_0}^Y$ and we use the tower property of the conditional expectations to achieve
	\begin{align*}
		\E[\Delta_t^{v,X,Y} \mid v_{t_0}^X,v_{t_0}^Y]&=\Delta_{t_0}^{v,X,Y}-k_v\int_{t_0}^t \E[\Delta_{s}^{v,X,Y} \mid v_{t_0}^X, v_{t_0}^Y]ds
	\end{align*}
	that implies
	\begin{equation*}
		\E[\Delta_t^{v,X,Y} \mid v_{t_0}^X,v_{t_0}^Y]=\Delta_{t_0}^{v,X,Y}e^{-k_v(t-t_0)}.
	\end{equation*}
	Now let us define the following stopping time and events
	\begin{equation}\label{eq:nulltime}
		\tau_0^v:=\inf\{t \ge t_0: \ \Delta_t^{v,X,Y}=0\} \ \mbox{ and } \ E_{\pm}=\{\omega \in \Omega: \ \pm v_{t_0}^X(\omega)>v_{t_0}^Y(\omega)\}.
	\end{equation}
	Recalling that over $E_\pm \cap \{\tau_0^v>t\}$ it holds $\pm \Delta_t^{v,X,Y}>0$, while over $\{\tau_0^v \le t\}$ it holds $\Delta_t^{v,X,Y}=0$, we have
	\begin{align*}
		\E[|\Delta_t^{v,X,Y}| \mid v_{t_0}^X,v_{t_0}^Y]&=\E[\Delta_t^{v,X,Y}\mathbf{1}_{E_+ \cap \{\tau_0^v>t\}} \mid v_{t_0}^X,v_{t_0}^Y]+\E[(-\Delta_t^{v,X,Y})\mathbf{1}_{E_- \cap \{\tau_0^v>t\}} \mid v_{t_0}^X,v_{t_0}^Y]\\
		&\le \mathbf{1}_{E_+}\E[\Delta_t^{v,X,Y} \mid v_{t_0}^X,v_{t_0}^Y]+\mathbf{1}_{E_-}\E[(-\Delta_t^{v,X,Y}) \mid v_{t_0}^X,v_{t_0}^Y]\\
		&=(\mathbf{1}_{E_+}\Delta_{t_0}^{v,X,Y}+\mathbf{1}_{E_-}(-\Delta_{t_0}^{v,X,Y}))e^{k_v(t-t_0)}=|\Delta_{t_0}^{v,X,Y}|e^{k_v(t-t_0)},
	\end{align*}
 where we used the fact that $|\Delta_{t_0}^{v,X,Y}|= \pm \Delta_{t_0}^{v,X,Y}$ on $E_\pm$.
\end{proof}
Thanks to this estimate, we can prove the $p$-weak ergodicity of $\{v_t\}_{t \ge 0}$ and we can provide a bound on the weak ergodicty rate.
\begin{thm}\label{thm:weakergovp}
Let $p \ge 1$, $X \in L^{2p-1}(\Omega; \R_{+,0}) \cap L^2(\Omega;\R_{+,0})$ be $\cF_0$-measurable and $Y$ be a $\cF_0$-measurable random variable with ${\rm Law}(Y)=\mathbf{m}_v$. Then $\{v_t^X\}$ is $p$-weakly ergodic and, setting $\mu_t^{X,v}:={\rm Law}(v_t^X)$,
\begin{equation}\label{eq:Wp}
	\cW_p(\mu_t^{X,v},\mathbf{m}_v) \le \left(\E\left[|\Delta_t^{v,X,Y}|^p\right]\right)^{\frac{1}{p}}\le (C_p^v(\E[X^p]))^{\frac{1}{p}}e^{-\frac{p-1}{p(\lceil p \rceil -1)}k_vt}
\end{equation}
where $C_p^v$ is a non-decreasing positive function depending only on $k_v,\eta,\gamma$ and $p$ and we set, for $p=1$, $\frac{p-1}{p(\lceil p \rceil -1)}:=1$.
\end{thm}
\begin{proof}
Let us stress that the first inequality in \eqref{eq:Wp} follows by the definition of $p$-Wasserstein distance, once we observe that $(v_t^X,v_t^Y) \in C(\mu_t^{X,v},\mathbf{m}_v)$.

First, we prove the statement for $p=1$. Let $Y$ be a random variable such that ${\rm Law}(Y)=\bm{m}_v$. Recall that the moments of $Y$ are given by \eqref{eq:momGamma}. Observe that, by definition of Wasserstein distance and Lemma~\ref{lem:boundcondexp}, for any $t \ge 0$,
	\begin{equation*}
		\cW_1({\rm Law}(v_t^X),\bm{m}_v) \le \E\left[|\Delta^{v,X,Y}_t|\right] \le \E[|X-Y|]e^{-k_vt}
	\end{equation*}
        that leads to $\eqref{eq:Wp}$ for $p=1$.

Next, before proving the statement for $p>1$, we show the following identity:

\begin{align}\label{eq:overDeltapostlim2}
		\begin{split}
			e^{pk_vt}\E[|\Delta_{t}^{v,X,Y}|^p]\le \E[|\Delta_0^{v,X,Y}|^p]+\frac{p(p-1)\gamma^2}{2}\int_{0}^{t}e^{pk_vs}\E\left[|\Delta_{s}^{v,X,Y}|^{p-1}\right]ds.
		\end{split}
\end{align}
To do this let $\overline{\Delta}_t^{v,X,Y}=e^{pk_vt}\Delta_t^{v,X,Y}$. Let $M>0$ and consider the stopping times $\tau_0^v$ as in \eqref{eq:nulltime},
	\begin{equation*}
		T_M^v:=\inf\{t \ge 0: \ |\Delta_t^{v,X,Y}| \le M\},
	\end{equation*}
 and $\tau^{v,M}_t=t \wedge T_M^v \wedge \tau_0^v$.
 
    By It\^o's formula, we have
	\begin{align}\label{eq:overDeltaprelim}
		\begin{split}
			&	|\overline{\Delta}_{\tau_t^{v,M}}^{v,X,Y}|^p=|\Delta_{0}^{v,X,Y}|^p+\frac{p(p-1)\gamma^2}{2}\int_{0}^{\tau_t^{v,M}}e^{2k_vs}|\overline{\Delta}_s^{v,X,Y}|^{p-2}(\sqrt{v_s^X}-\sqrt{v_s^Y})^2ds\\
			&\qquad +p\gamma \int_{0}^{\tau_t^{v,M}}\overline{\Delta}_{\tau_s^{v,M}}^{v,X,Y}|\overline{\Delta}_{\tau_s^{v,M}}^{v,X,Y}|^{p-2}e^{k_vs}\left(\sqrt{v_{\tau_s^{v,M}}^X}-\sqrt{v_{\tau_s^{v,M}}^Y}\right)dW_s^{(3)}.	 		
		\end{split}
	\end{align}
        Since $2p-2>0$, for any $t \ge 0$ we have
	\begin{align*}
		\E&\left[\left|\overline{\Delta}_{\tau_t^{v,M}}^{v,X,Y}\right|^{2p-2}e^{2k_vt}\left(\sqrt{v_{\tau_t^{v,M}}^X}-\sqrt{v_{\tau_t^{v,M}}^Y}\right)^2\right] \\
		&\quad\le e^{2pk_vt}\left(\E\left[|\overline{\Delta}_{\tau_t^{v,M}}^{v,X,Y}|^{2p-1}\mathbf{1}_{\{T_M=0\}}\right]+\E\left[|\overline{\Delta}_{\tau_t^{v,M}}^{v,X,Y}|^{2p-1}\mathbf{1}_{\{T_M>0\}}\right]\right)\\
		&\quad\le e^{2pk_vt}\left(2^{2p-2}\left(\E[X^{2p-1}]+M_{2p-1}(\mathbf{m}_v)\right)+M^{2p-1}\right), 
	\end{align*}
        it is clear that the integrand in the stochastic integral on the right-hand side of \eqref{eq:overDeltaprelim} belongs to $M_t^2(\Omega)$ for all $t \ge 0$, hence, taking the expectation in \eqref{eq:overDeltaprelim}, by the optional stopping theorem, we get
        \begin{align}\label{eq:overDeltaprelim2}
		\begin{split}
			&	\E[|\overline{\Delta}_{\tau_t^{v,M}}^{v,X,Y}|^p]=\E[|\Delta_{0}^{v,X,Y}|^p]\\
			&\qquad 
   +\E\left[\int_{0}^{\tau_t^{v,M}}\frac{p(p-1)\gamma^2}{2}e^{2k_vs}|\overline{\Delta}_s^{v,X,Y}|^{p-2}\left(\sqrt{v_s^X}-\sqrt{v_s^Y}\right)^2ds\right].
		\end{split}
	\end{align}
        Since $T_M \to +\infty$, we know that $\tau_t^{v,M} \to t \wedge \tau_0^v$. Furthermore, by pathwise uniqueness of the solution of \eqref{eq:v}, it is clear that if $t \ge \tau_0$ then $\overline{\Delta}_t^{v,X,Y}=0$ and then
        \begin{equation*}
        \E[|\overline{\Delta}_{t}^{v,X,Y}|^p]=\E[|\overline{\Delta}_{t \wedge \tau_0^v}^{v,X,Y}|^p\mathbf{1}_{\{\tau_0^v > t\}}] \le \E[|\overline{\Delta}_{t \wedge \tau_0^v}^{v,X,Y}|^p].
	\end{equation*}
        Hence, taking the limit as $M \to +\infty$ in \eqref{eq:overDeltaprelim2} and using both Fatou's lemma and monotone convergence theorem, we finally get \eqref{eq:overDeltapostlim2}.

        Now we are ready to prove the sencond inequalty in \eqref{eq:Wp}. First, let us show it for $p \in \N$ by induction. We already proved it for $p=1$, hence assume that this is true for $p-1 \in \N$. By \eqref{eq:overDeltapostlim2} we get
        \begin{align*}
		e^{pk_vt}\E[|\Delta_{t}^{v,X,Y}|^p] 
		&\le 2^{p-1}(\E[X^p]+M_p(\mathbf{m}_v))+\frac{p(p-1)\gamma^2C^v_{p-1}(X)}{2(p-1)k_v}e^{(p-1)k_vt}
	\end{align*}
	so that, multiplying both sides by $e^{-pk_vt}$ and using H\"older's inequality, we get \eqref{eq:preWp}, with
	\begin{equation*}
		C^v_p(x)=2^{p-1}(x+M_p(\mathbf{m}_v))+\frac{p(p-1)\gamma^2C^v_{p-1}(x)}{2(p-1)k_v}, \ x \ge 0.
	\end{equation*}
    This proves \eqref{eq:Wp} for any $p \in \N$.
    
    Next, let $p \in (1,2)$. Then $p-1<1$ and we can use Jensen's inequality and \eqref{eq:Wp} for $p=1$ to get
    \begin{equation*}
		\E\left[|\Delta_{s}^{v,X,Y,\varepsilon}|^{p-1}\right] \le (\E\left[|\Delta_{s}^{v,X,Y,\varepsilon}|\right])^{p-1}\le (\E[X^p]+\eta)^{p-1}e^{-(p-1)k_vs}.
	\end{equation*}
	Hence, by \eqref{eq:overDeltapostlim2}
	\begin{align}
		\begin{split}
			e^{pk_vt}\E[|\Delta_{t}^{v,X,Y}|^p] \le \E[X^p]+M_p(\mathbf{m}_v)+\frac{p(p-1)\gamma^2}{2k_v}(\E[X^p]+\eta)^{p-1}e^{k_vt},
		\end{split}
	\end{align}
    that implies \eqref{eq:Wp} once we multiply both sides of the inequaity by $e^{-pk_vt}$.

    Finally, let $p>2$ with $p \not \in \N$. We already know that the \eqref{eq:Wp} holds for $\lfloor p \rfloor \in \N$. Furthermore, by H\"older's inequality with exponent $\frac{\lfloor p \rfloor}{p-1}>1$, we get
	\begin{equation*}
		\E\left[|\Delta_{s}^{v,X,Y,\varepsilon}|^{p-1}\right] \le (\E\left[|\Delta_{s}^{v,X,Y,\varepsilon}|^{\lfloor p \rfloor}\right])^{\frac{p-1}{\lfloor p \rfloor}}\le (C^v_{\lfloor p \rfloor}(\E[X^p]))^\frac{p-1}{\lfloor p \rfloor}e^{-\frac{p-1}{\lfloor p \rfloor}k_vs},
	\end{equation*}
        where we also used the fact that $C_{\lfloor p \rfloor}$ is non-decreasing and $\E[X^{\lfloor p \rfloor}] \le \E[X^p]$. By \eqref{eq:overDeltapostlim2} we get
	\begin{align*}
		\E[|\overline{\Delta}_{t}^{v,X,Y}|^p] 
		&\le 2^{p-1}(\E[X^p]+M_p(\mathbf{m}_v))+\frac{p(p-1)\gamma^2(C^v_{\lfloor p \rfloor}(\E[X^p]))^\frac{p-1}{\lfloor p \rfloor}\lfloor p \rfloor}{2(p(\lfloor p \rfloor-1)+1)}e^{p-\frac{p-1}{\lfloor p \rfloor}k_vt}
	\end{align*}
        that implies \eqref{eq:Wp} with 
        \begin{equation*}
		C^v_p(x)=2^{p-1}(x+M_p(\mathbf{m}_v))+\frac{p(p-1)\gamma^2(C^v_{\lfloor p \rfloor}(x))^\frac{p-1}{\lfloor p \rfloor}\lfloor p \rfloor}{2(p(\lfloor p \rfloor-1)+1)}.
	\end{equation*}

\end{proof}

\begin{remark}
 Let us recall that $\{v_t\}_{t \ge 0}$ is actually ergodic in total variation, by \cite[Theorem V.54.5]{rogers2000diffusions}. This implies the $1$-weak ergodicity, but does not provide the exponential bound in \eqref{eq:Wp}. Furthermore, let us recall that the total variation distance does not provide, in general, an upper bound for the $p$-Wasserstein distance as $p>1$, as observed, for instance, in \cite[Theorem 6.15]{villani2009optimal}. This implies that, despite the name, the \textit{strong} ergodicity in total variation does not imply in general the $p$-weak ergodicity if $p>1$. 
\end{remark}
In the following we will use the moments $F_{p}^{v,X}(t):=\E[(v_t^{X})^p]$ of $\{v_t^X\}_{t \ge 0}$. We obtain, as a further consequence of Theorem \ref{thm:weakergovp}, the following uniform bounds, together with a regularity estimate.
\begin{cor}\label{cor:upboundFpv}
	Let $p \ge 1$ and $X \in L^{2p-1}(\Omega; \R_{+,0}) \cap L^2(\Omega; \R_{+,0})$ be $\cF_0$-measurable. Then 
	\begin{equation}\label{eq:controllocal}
		\sup_{t \ge 0}F_{p}^{v,X}(t) \le \widetilde{C}_p^v(\E[X^p]), \ \mbox{ where } \
		\widetilde{C}_p^v(x)=2^{p-1}(C_p^v(x)+M_p(\bm{m}_v)).
	\end{equation}
    Furthermore, if $X \in L^{4p-3}(\Omega; \R_+) \cap L^2(\Omega;\R_+)$, then $F_p^{v,X} \in C^{\lfloor p \rfloor-1}(\R_+)$. If, in particular $p \in \N$, then $F_p^{v,X} \in C^\infty(\R_+)$.
\end{cor}
\begin{proof}
	Let $X \in L^{2p-1}(\Omega;\R_+)\cap L^2(\Omega;\R_+)$ and $Y$ be $\cF_0$-measurable with ${\rm Law}(Y)=\mathbf{m}_v$. We have, for any $t \ge 0$,
	\begin{equation*}
		F_p^{v,X}(t) \le 2^{p-1}\left(\E\left[\left|\Delta_t^{v,X,Y}\right|^p\right]+M_p(\mathbf{m}_v)\right)  \le 2^{p-1}\left(C^v_p(\E[X^p])+M_p(\mathbf{m}_v)\right)<\infty.
	\end{equation*}
        
        Now assume further that $X \in L^{4p-3}(\Omega;\R_+)$. The fact that $F_1^{v,X} \in C^\infty(\R_+)$ follows by \eqref{eq:expecvt}. Now let $p>1$ and consider for $\varepsilon \in (0,1)$ the function $V_{p,\varepsilon}(x)=(x^2+\varepsilon)^{\frac{p}{2}}$ for $x \in \R$. This function belongs to $C^2(\R)$ hence we can use It\^{o}'s formula to write
        \begin{align*}
        	&V_{p,\varepsilon}(v_t^X)=(X^2+\varepsilon)^{\frac{p}{2}}\\
        	&\quad +p\int_0^t(|v_s^X|^2+\varepsilon)^{\frac{p}{2}-2}\left[k_v(|v_s^X|^2+\varepsilon)v^X_s(\eta-v^X_s)+\frac{\gamma^2}{2}\left((p-1)|v_s^X|^2+\varepsilon\right)v^X_s\right]ds\\
        	&\quad +\gamma p\int_0^t(|v_s^X|^2+\varepsilon)^{\frac{p}{2}-1}|v_s^X|^{\frac{3}{2}}dW^{(3)}_s.
        \end{align*}
    	Since $X \in L^{4p-3}(\Omega;\R_+) \cup L^2(\Omega;\R_+)$, we know by the first part of the statement that $\sup_{t \ge 0}F_{2p-1}^{v,X}(t)<\infty$ and then it is not difficult to check that $\{(|v_s^X|^2+\varepsilon)^{\frac{p}{2}-1}|v_s^X|^{\frac{3}{2}}\}_{s \ge 0} \in M^2_t(\Omega)$. Furthermore, observe that the first integrand can be bounded as follows:
    	\begin{align*}
    		&(|v_s^X|^2+\varepsilon)^{\frac{p}{2}-2}\left|k_v(|v_s^X|^2+\varepsilon)v^X_s(\eta-v^X_s)+\frac{\gamma^2}{2}\left((p-1)|v_s^X|^2+\varepsilon\right)v^X_s\right|\\
    		&\quad  \le (|v_s^X|^2+\varepsilon)^{\frac{p-1}{2}} \left[\left(k_v\eta+\frac{\gamma^2}{2}\max\{p-1,1\}\right)+k_v\sqrt{|v_s^X|^2+\varepsilon}\right]\\
    		&\quad\le (|v_s^X|^2+1)^{\frac{p-1}{2}} \left[\left(k_v\eta+\frac{\gamma^2}{2}\max\{p-1,1\}\right)+k_v\sqrt{|v_s^X|^2+1}\right]\\
    		&\quad \le C(|v_s^X|^{p-1}+|v_x^X|^p),
    	\end{align*}
    	where $C$ is a constant depending only on $p,k_v,\gamma$ and $\eta$. Hence, since $F_p^{v,X}(t)$ is bounded, we can take the expectation and use Fubini's theorem to get 
    	\begin{align*}
    		&\E\left[V_{p,\varepsilon}(v_t^X)\right]=\E\left[(X^2+\varepsilon)^{\frac{p}{2}}\right]\\
    		&\quad +p\int_0^t\E\left[(|v_s^X|^2+\varepsilon)^{\frac{p}{2}-2}\left[k_v(|v_s^X|^2+\varepsilon)v^X_s(\eta-v^X_s)+\frac{\gamma^2}{2}\left((p-1)|v_s^X|^2+\varepsilon\right)v^X_s\right]\right]ds.
    	\end{align*}
        Taking the limit as $\varepsilon \to 0$ and using the dominated convergence theorem, we finally achieve
%
%
%
	\begin{equation}\label{eq:prederivative}
		F_p^{v,X}(t)=\E[X^p]+\int_0^t\left(\left(pk_v\eta+\frac{p(p-1)}{2}\gamma\right)F_{p-1}^{v,X}(s)-pk_vF_p^{v,X}(s)\right)ds,
        \end{equation}
    where the right-hand side is absolutely continuous. This proves that $F_p^{v,X}$ is continuous. To show that $F_p^{v,X} \in C^{\lfloor p \rfloor}(\R_+)$ let us argue by induction on $\lfloor p \rfloor$. We already proved the statement for $\lfloor p \rfloor=1$. Assume that $F_p^{v,X} \in C^{n-1}(\R_+)$ if $p \in [n,n+1)$ and consider any $p \in [n+1,n+2)$. Then $F_{p-1}^{v,X} \in C^{n-1}(\R_+)$ and, with a simple bootstrap procedure, \eqref{eq:prederivative} implies $F_p^{v,X} \in C^n(\R_+)$. Finally, if $p \in \N$, we still proceed by induction. Indeed, we already know that $F_1^{v,X} \in C^\infty(\R_+)$. Furthermore, if $F_p^{v,X} \in C^\infty(\R_+)$, then \eqref{eq:prederivative} clearly implies that $F_{p+1}^{v,X} \in C^\infty(\R_+)$.
\end{proof}
\begin{remark}
	Under Feller's condition $2k_v\eta>\gamma^2$, we know that $\{v_t^X\}_{t \ge 0}$ cannot touch $0$ and then the previous statement could be proved by using It\^{o}'s formula on $|v_t^X|^p$. This is not possible if $p \in (1,2)$ an $2k_v\eta \le \gamma^2$.
\end{remark}
Now we are ready to move forward and prove the $p$-weak ergodicity of the two-factors process $\{\bm{\theta}_t\}_{t \ge 0}$.

\section{The $p$-weak ergodicity of the two-factor process}\label{subsec:twofactor}

Now we want to prove the $p$-weak ergodicity of the two-factor process $\{\theta_t\}_{t \ge 0}$. Let us first underline that classical results cannot be used in this specific context. For instance, strong ergodicity (i.e., ergodicity in total variation) of $\{\bm{\theta}_t\}_{t \ge 0}$ cannot be shown through coupling methods as in \cite[Theorem V.54.5]{rogers2000diffusions}, since the process takes value in a subset of $\R^2$ and thus, despite being a diffusion, it is not $1$-dimensional. On the other hand, we cannot even employ such a result for just the second factor $\{\theta_t\}_{t \ge 0}$, since it solves a non-autonomous SDE, as the diffusion term depends explicitly on $\{v_t\}_{t \ge 0}$. We focus in particular on $p$-weak ergodicity since, as we will see in what follows, the $p$-th order moments of the two-factor process $\{\bm{\theta}_t\}_{t \ge 0}$ are still sufficiently manageable. However, even in this case, classical weak ergodicity results cannot be used to prove directly the $p$-weak ergodicity of the process. For instance, we are not able to apply directly the general Dobrushin theorem \cite[Theorem 4.5.1]{kulik2017ergodic}. In general, for dissipative systems with Lipschitz coefficients, $p$-weak exponential ergodicity can be shown with a simple employment of It\^o formula (see the discussion in \cite[Sections 4.2.1 and 4.2.2]{kulik2017ergodic}. However, despite \eqref{Our_Model_true} being a dissipative system, the lack of the Lipschitz property of the diffusion term adds a further complication in proving that we are under the assumptions of Dobrushin theorem. On the other hand, following in detail the proof of \cite[Proposition 4.2.1]{kulik2017ergodic}, it is clear that, in the one-dimensional case, the diffusion term only needs to be $1/2$-H\"older continuous. However, again, we cannot apply such an argument to the process $\{\theta_t\}_{t \ge 0}$ since despite the diffusion term is $1/2$-H\"older continuous, it is explicitly dependent on time via $\{v_t\}_{t \ge 0}$. For this reason, we will use a more intricate strategy to prove the $p$-weak ergodicity of $\{\bm{\theta}_t\}$.


Let us first introduce some notations. For any $\mathbf{X}=(X^{(1)},X^{(2)}) \in L^2(\Omega;\R_{+,0}^2)$ $\cF_0$-measurable, we let $\{\bm{\theta}_t^\mathbf{X}\}_{t \ge 0}=\{(\theta_t^{\mathbf{X}},v_t^{\mathbf{X}})\}_{t \ge 0}$ be the pathwise unique strong solution of 
\begin{equation}\label{eq:SDEaux2a}
	\begin{cases}
		d\theta_t^{\mathbf{X}}=k_\theta(\zeta-\theta^{\mathbf{X}}_t)dt+\alpha\beta \sqrt{v^{\mathbf{X}}_t\theta^{\mathbf{X}}_t}dW^{(2)}_t\\
		dv^{\mathbf{X}}_t=k_v(\eta-v^{\mathbf{X}}_t)dt+\gamma \sqrt{v^{\mathbf{X}}_t}dW^{(3)}_t\\
		\theta^{\mathbf{X}}_{0}=X^{(1)} \qquad v^{\mathbf{X}}_0=X^{(2)}.
	\end{cases}
\end{equation}
For any $p \ge 0$, we set $G_{p}^{\theta,\mathbf{X}}(t)=\E[|\theta_t^{\mathbf{X}}|^pv_t^{\mathbf{X}}]$ and $F_p^{\theta,\mathbf{X}}(t)=\E[|\theta_t^{\mathbf{X}}|^p]$. Finally, we define $\Delta_t^{\theta,\mathbf{X},\mathbf{Y}}=\theta_t^{\mathbf{X}}-\theta_t^{\mathbf{Y}}$ for any $\mathbf{X},\mathbf{Y} \in L^2(\Omega;\R_{+,0}^2)$ $\cF_0$-measurable and $t \ge 0$. 

Precisely, we will argue as follows:
\begin{itemize}
	\item[(I)] First we provide some uniform bounds on the functions $F_p^{\theta,\mathbf{X}}$ under some assumptions on $\mathbf{X}$. This will be done in three substeps: first, we prove some local bounds on $F_p^{\theta,\mathbf{X}}$, then we do the same on $G_p^{\theta,\mathbf{X}}$ and finally we prove the global bound on $F_p^{\theta,\mathbf{X}}$.
	\item[(II)] Next we prove that $\{\bm{\theta}_t^{\mathbf{X}}\}_{t \ge 0}$ is $p$-weakly ergodic under the assumptions \linebreak ${\rm Law}(X^{(1)})=\mathbf{m}_v$ and $X^{(2)} \in L^{p^\prime}(\Omega;\R_{+,0})$ with $p^\prime>4p-3$. To do this, we first provide a bound on the expectation of $\left|\Delta_t^{\theta,\mathbf{X},\mathbf{Y}}\right|$, then we prove the $1$-weak ergodicity by extending the proof of Dobrushin's theorem to our case. Finally, we prove $p$-weak ergodicity by using the uniform bounds we proved before.
	\item[(III)] As a third step, we need to study some properties of the limit distribution $\bm{\pi}^{v,\theta}$ we found in Step $(II)$. Precisely, we prove that this is an invariant distribution and if, for any fixed initial data $\mathbf{X}$, $\{\bm{\theta}_t^{\mathbf{X}}\}_{t \ge 0}$ admits a limit distribution, such a limit distribution must be $\bm{\pi}^{v,\theta}$. This also tells us that $\bm{\pi}^{v,\theta}$ is the unique invariant distribution.
	\item[(IV)] Finally, we use the uniform bounds on $F_p^{\theta,\mathbf{X}}$ provided in Step (I) to show that $\{\bm{\theta}_t^{\mathbf{X}}\}_{t \ge 0}$ is $p$-weakly ergodic under suitable integrability assumptions on the initial data $\mathbf{X}$ (that are clearly satisfied by any degenerate/deterministic initial data).
\end{itemize}

\subsection{Step (I): global bounds on the moments}\label{subs:S1}
We first want to prove that, under suitable conditions on $\mathbf{X}$, $\sup_{t \ge 0}F_p^{\theta, \mathbf{X}}(t)<\infty$. To do this, we first need to provide some local bounds on $F_p^{\theta, \mathbf{X}}$ and $G_p^{\theta, \mathbf{X}}$. 
\begin{lem}\label{lem:Fptheta}
	Let $p \ge 2$, $p^\prime>2p-1$ and $\widetilde{p}=\max\left\{2p-1,\frac{p^\prime}{p^\prime-2p+1}\right\}$. Consider also two $\cF_0$-measurable random variables $X^{(1)} \in L^{p^\prime}(\Omega;\R_{+,0})$ and $X^{(2)} \in L^{\widetilde{p}}(\Omega;\R_{+,0})$. 
	Then, for any $T>0$,
	\begin{equation*}
		\sup_{t \in [0,T]}F_p^{\theta,\mathbf{X}}(t)<\infty.
	\end{equation*}
	If, furthermore, $X^{(1)}$ and $X^{(2)}$ are independent, then the statement holds even for $p^\prime=\widetilde{p}=2p-1$.	
\end{lem}
\begin{proof}
	
	Let $p \ge 2$ and observe that, by It\^o's formula,
	\begin{align}\label{eq:thetap2}
		\begin{split}
			(\theta_{\tau^{\theta,\bm{X},M}_t}^{\mathbf{X}})^p&=\left|X^{(1)}\right|^p+\int_0^{\tau^{\theta,\bm{X},M}_t} \left(-k_\theta p \left|\theta^{\mathbf{X}}_s\right|^p+\left(k_\theta p \zeta+\frac{\alpha^2\beta^2p(p-1)}{2}v_s^{\mathbf{X}}\right)|\theta_s^{\mathbf{X}}|^{p-1}\right)ds\\
			&\qquad +\alpha \beta p\int_0^{\tau^{\theta,\bm{X},M}_t} \sqrt{v_{\tau^{\theta,\bm{X},M}_s}^{\mathbf{X}}}|\theta_{\tau^{\theta,\bm{X},M}_s}^{\mathbf{X}}|^{p-\frac{1}{2}}dW_s^{(2)},
		\end{split}
	\end{align}
	where, for any $M,t>0$,
	\begin{equation}\label{eq:TRtheta}
		T_M^{\theta,\mathbf{X}}:=\inf\{t \ge 0: \ \theta_t^{\mathbf{X}} \vee v_t^{\mathbf{X}}   \ge  M\}, \quad  \tau^{\theta,\bm{X},M}_t=t \wedge T_M^{\theta,\bm{X}}.
	\end{equation}
		Once we notice that
	\begin{align}\label{eq:Mt2theta1}
		\begin{split}
		\E&\left[v_{\tau^{\theta,\bm{X},M}_s}^{\mathbf{X}}|\theta_{\tau^{\theta,\bm{X},M}_s}^{\mathbf{X}}|^{2p-1}\right]\\
		&=\E\left[v_{\tau^{\theta,\bm{X},M}_s}^{\mathbf{X}}|\theta_{\tau^{\theta,\bm{X},M}_s}^{\mathbf{X}}|^{2p-1}\mathbf{1}_{\{T_M^{\theta,\mathbf{X}}=0\}}\right]+\E\left[v_{\tau^{\theta,\bm{X},M}_s}^{\mathbf{X}}|\theta_{\tau^{\theta,\bm{X},M}_s}^{\mathbf{X}}|^{2p-1}\mathbf{1}_{\{T_M^{\theta,\mathbf{X}}>0\}}\right]\\
		& \le \E\left[X^{(2)}\left|X^{(1)}\right|^{2p-1}\right]+M^{2p} \le \left(\E\left[\left|X^{(2)}\right|^{\frac{p^\prime}{p^\prime-2p+1}}\right]\right)^{\frac{p^\prime-2p+1}{p^\prime}}\E\left[\left|X^{(1)}\right|^{p^\prime}\right]^{\frac{2p-1}{p^\prime}}+M^{2p},	
		\end{split}
	\end{align}
	we know that $\left\{\sqrt{v_{\tau^{\theta,\bm{X},M}_s}^{\mathbf{X}}}|\theta_{\tau^{\theta,\bm{X},M}_s}^{\mathbf{X}}|^{p-\frac{1}{2}}\right\}_{s \ge 0} \in M^2_t(\Omega)$ for any $t \ge 0$, hence we can take the expectation in \eqref{eq:thetap2} to achieve, by the optional stopping theorem,
	\begin{align}\label{eq:thetap3}
		\begin{split}
			\E\left[(\theta_{\tau^{\theta,\bm{X},M}_t}^{\mathbf{X}})^p\right]&=\E\left[\left|X^{(1)}\right|^p\right]\\
			&+\E\left[\int_0^{\tau^{\theta,\bm{X},M}_t} \left(-k_\theta p \left|\theta^{\mathbf{X}}_s\right|^p+\left(k_\theta p \zeta+\frac{\alpha^2\beta^2p(p-1)}{2}v_s^{\mathbf{X}}\right)|\theta_s^{\mathbf{X}}|^{p-1}\right)ds\right].
		\end{split}
	\end{align}
	Now set
	\begin{equation*}
		\varepsilon=\left(\frac{2 k_\theta p}{\alpha^2\beta^2(p-1)^2+2k_\theta\zeta(p-1)}\right)^{1-\frac{1}{p}}
	\end{equation*}
	and use Young's inequality to write
	\begin{equation}\label{eq:Youngp1}
		\left|\theta_s^{\mathbf{X}}\right|^{p-1} \le \frac{1}{p\varepsilon^p}+\frac{(p-1)\varepsilon^{\frac{p}{p-1}}}{p}\left|\theta_s^{\mathbf{X}}\right|^{p}, \quad  v_s^{\mathbf{X}}\left|\theta_s^{\mathbf{X}}\right|^{p-1} \le \frac{\left|v_s^{\mathbf{X}}\right|^p}{p\varepsilon^p}+\frac{(p-1)\varepsilon^{\frac{p}{p-1}}}{p}\left|\theta_s^{\mathbf{X}}\right|^{p}.
	\end{equation}
	Combining these inequalities with \eqref{eq:thetap3}, we achieve
	\begin{align}\label{eq:thetap5}
		\begin{split}
			\E&\left[(\theta_{\tau^{\theta,\bm{X},M}_t}^{\mathbf{X}})^p\right]=\E\left[\left|X^{(1)}\right|^p\right]\\
			&+\left(\frac{\alpha^2\beta^2(p-1)^2+2k_\theta\zeta(p-1)}{2pk_\theta}\right)^{p-1}\E\left[\int_0^{\tau^{\theta,\bm{X},M}_t} \left(\zeta k_\theta +\frac{\alpha^2\beta^2(p-1)}{2}\left|v_s^{\mathbf{X}}\right|^p\right)ds\right],
		\end{split}
	\end{align}
	that, in turn, taking the limit as $M \to +\infty$ and using Fatou's lemma, leads to
	\begin{align}\label{eq:thetap6}
		\begin{split}
			&F_p^{\theta,\mathbf{X}}(s)\le \E\left[\left|X^{(1)}\right|^p\right]\\
			&\quad +\left(\frac{\alpha^2\beta^2(p-1)^2+2k_\theta\zeta(p-1)}{2pk_\theta}\right)^{p-1}\left(\zeta k_\theta+\frac{\alpha^2\beta^2(p-1)}{2}C_p^{v}(\E[|X^{(2)}|^p])\right)t,
		\end{split}
	\end{align}
	where $C_p^{v}$ is defined in \ref{cor:upboundFpv}.
	
	Finally, to prove the statement under the additional condition that $X^{(1)}$ and $X^{(2)}$ are independent, it is sufficient to observe that, arguing as in \eqref{eq:Mt2theta1}
	\begin{equation*}
		\E[v_{\tau_s^{\theta,\bm{X},M}}^{\bm{X}}|\theta_{\tau_s^{\theta,\bm{X},M}}^{\bm{X}}|^{2p-1}] \le \E[X^{(2)}]\E[|X^{(1)}|^{2p-1}]+M^{2p}<\infty
	\end{equation*}
	and then the proof proceeds exactly in the same way.
\end{proof}
As a direct consequence of Lemma~\ref{lem:Fptheta}, we get the following local bounds on $G_p^{\theta,\bm{X}}$.
\begin{lem}\label{lem:upGptheta}
	Let $p \ge 2$, $\overline{p}>p$, $p^\prime>2\overline{p}-1$ and
	$\widetilde{p}=\max\left\{2\overline{p}-1,\frac{p^\prime}{p^\prime-2\overline{p}+1}\right\}$. Consider also two $\cF_0$-measurable random variables $X^{(1)} \in L^{p^\prime}(\Omega;\R_{+,0})$ and $X^{(2)} \in L^{\widetilde{p}}(\Omega;\R_{+,0})$. Then, for any $T>0$
	\begin{equation*}
		\sup_{t \in [0,T]}G_p^{\theta,\mathbf{X}}(t)<\infty.
	\end{equation*}	
\end{lem}
\begin{proof}
	Without loss of generality, we can assume that $p^\prime \le \frac{(\overline{p}+p)(2\overline{p}-1)}{2p}$, so that $\frac{\overline{p}+p}{\overline{p}-p} \le \frac{p'}{p'-2\overline{p}+1}$, and then $\frac{\overline{p}}{\overline{p}-p}$ satisfies the assumptions of Corollary~\ref{cor:upboundFpv}. At the same time, $\overline{p}$ satisfies the assumptions of Lemma~\ref{lem:Fptheta}. The statement follows by Lemma~\ref{lem:Fptheta} and Corollary~\ref{cor:upboundFpv} once we observe that, by Young's inequality,
	\begin{equation*}
		G_p^{\theta,X}(t) \le \frac{p}{\overline{p}}F_{\overline{p}}^{\theta,X}(t)+\frac{\overline{p}-p}{\overline{p}}F_{\frac{\overline{p}}{\overline{p}-p}}^{v,X^{(2)}}(t).
	\end{equation*}
\end{proof}
Next, let us recall the following really easy lemma on linear integral inequalities.
\begin{lem}\label{lem:contbound}
	Let $F \in C(\R_{+,0})$ and assume there exist two constants $A,B>0$ such that for any $0 \le t_0<t$,
	\begin{equation}\label{eq:integralineq}
		F(t) \le F(t_0)+\int_{t_0}^t(A-BF(s))ds, \ \forall t \ge 0.
	\end{equation}
	Then
	\begin{equation*}
		F(t) \le \max\left\{F(0), \frac{A}{B}\right\}, \ \forall t \ge 0.
	\end{equation*}
\end{lem}
\begin{proof}
	Set $C=\max\left\{F(0), \frac{A}{B}\right\}$ and assume by contradiction that there exists $t^\ast$ such that $F(t^\ast)>C$. Then, since $F(0) \le C$, $t^\ast>0$ and we can define, by continuity of $F$,
	\begin{equation*}
		\overline{t}:=\inf\left\{t \in [0,t^\ast], \ F(s)>C,\ \forall s \in (t,t^\ast)\right\}.
	\end{equation*}
	By definition, for $t \in (\overline{t},t^\ast)$ we have $F(t)>C$, while $F(\overline{t})=C$. However, by \eqref{eq:integralineq} we have
	\begin{equation*}
		F(t^\ast) \le F(\overline{t})+\int_{\overline{t}}^t(A-BF(s))ds<C
	\end{equation*}  
	that is absurd.
\end{proof}
Now we are ready to prove the global bounds we were searching for
\begin{thm}\label{lem:Fpthetaglob}
	Let $p \ge 2$, $\overline{p}>2p-1$, $p^\prime>2\overline{p}-1$ and $\widetilde{p}=\max\left\{2\overline{p}-1,\frac{p^\prime}{p^\prime-2\overline{p}+1}\right\}$. Consider two $\cF_0$-measurable random variables $X^{(1)} \in L^{p^\prime}(\Omega;\R_{+,0})$ and $X^{(2)} \in L^{\widetilde{p}}(\Omega;\R_{+,0})$. Then $F_p^{\theta,X}$ is continuous and
	\begin{equation}\label{eq:ineqglobaltheta}
		\sup_{t \ge 0}F_p^{\theta,X}(t) \le \widetilde{C}_p^{\theta}(\E[|X^{(1)}|^p],\E[|X^{(2)}|^p]), \ \mbox{ where } \ \widetilde{C}_p^{\theta}(x_1,x_2)=\max\{x_1,C_p^{\theta}(x_2)\},
	\end{equation}
	and
	\begin{equation*}
		C_p^{\theta}(x_2)=\left(\frac{\alpha^2\beta^2(p-1)+2k_\theta\zeta (p-1)}{pk_\theta}\right)^{p-1}\left(\zeta k_\theta+\frac{\alpha^2\beta^2(p-1)}{2}\widetilde{C}_p^{v}(x_2)\right)\frac{2}{k_\theta p}.
	\end{equation*}
\end{thm}
\begin{proof}
	Let us first observe that the assumptions of \ref{lem:upGptheta} are satisfied with the exponent $2p-1$ in place of $p$. Hence we know that $\sup_{t \in [0,T]}G_{2p-1}^{\theta,X}(t)<\infty$ for any $T>0$ and then $\{\sqrt{v_s}\left(\theta_s^X\right)^{p-\frac{1}{2}}\}_{s \ge 0}\in M^2_t(\Omega)$. Hence, by using It\^o's formula as in \eqref{eq:thetap2} and taking the expectation, we get, for any $t_0 \ge 0$,
	\begin{equation}\label{eq:Fpeq3}
		F_p^{\theta,\mathbf{X}}(t)=F_p^{\theta,\mathbf{X}}(t_0)+\int_{t_0}^{t}\left(k_\theta p\zeta F_{p-1}^{\theta,\mathbf{X}}(s)+\frac{\alpha^2\beta^2p(p-1)}{2}G_{p-1}^{\theta,\mathbf{X}}(s)-k_\theta pF_{p}^{\theta,\mathbf{X}}(s)\right)ds,
	\end{equation}
	where we used Fubini's theorem since Lemmas \ref{lem:Fptheta} and \ref{lem:upGptheta} guarantee that $F_p^{\theta,\mathbf{X}}$, $F_{p-1}^{\theta,\mathbf{X}}$ and $G_{p-1}^{\theta,\mathbf{X}}$ belong to $L^1_{\rm loc}(\R_{+,0})$. The integral equation \eqref{eq:Fpeq3} already proves the continuity of $F_p^{\theta,\bm{X}}$. Now we use again \eqref{eq:Youngp1}, this time with
	\begin{equation*}
		\varepsilon=\left(\frac{k_\theta p}{\alpha^2\beta^2(p-1)^2+2k_\theta\zeta (p-1)}\right)^{\frac{p}{p-1}},
	\end{equation*}
	to get
	\begin{equation}\label{eq:Fpeq4}
		F_p^{\theta,\mathbf{X}}(t) \le F_p^{\theta,\mathbf{X}}(t_0)+\int_{t_0}^{t}\left(\frac{k_\theta p C_p^{\theta}(\E[|X^{(2)}|^p])}{2}-\frac{k_\theta p}{2}F_{p}^{\theta,\mathbf{X}}(s)\right)ds
	\end{equation}
	that, by Lemma \ref{lem:contbound}, implies \eqref{eq:ineqglobaltheta}.
\end{proof}

\subsection{Step (II): $p$-weak ergodicity for stationary $v_t$}\label{subs:S2}
Now we want to use Theorem \ref{lem:Fpthetaglob} to prove the $p$-weak ergodicity of $\{{\rm Law}(\theta_t^{\mathbf{X}},v_t^{\mathbf{X}})\}_{t \ge 0}$ at least under a special choice of $\mathbf{X}$. To do this, however, we first need a preliminary estimate on $\Delta^{\theta,\mathbf{X},\mathbf{Y}}_t$, which can be carried on in a similar way as in Lemma~\ref{lem:boundcondexp}. 
\begin{lem}\label{lem:boundcondexptheta}
	Consider the $\cF_0$-measurable random variables $X^{(1)},Y^{(1)} \in L^{p}(\Omega; \R_{+,0})$ for $p>3$ and $X^{(2)}\in L^{\widetilde{p}}(\Omega; \R_{+,0})$ with $\widetilde{p}=\max\left\{\frac{p}{p-3},3\right\}$. Define $\mathbf{X}=(X^{(1)},X^{(2)})$ and $\mathbf{Y}=(Y^{(1)},X^{(2)})$. For any $t \ge 0$ we have
	\begin{equation}\label{eq:condexptheta}
		\E\left[|\Delta_t^{\theta, \mathbf{X},\mathbf{Y}}|\right] \le \E\left[\left|X^{(1)}-Y^{(1)}\right|\right]e^{-k_\theta t}.
	\end{equation}
\end{lem}
\begin{proof}
	Observe that, by definition of $\bm{X},\bm{Y}$, it holds $v^{\bm{X}}_t=v^{\bm{Y}}_t$ a.s. for all $t \ge 0$. Hence,
	\begin{align}\label{eq:firstmomdtheta1}
		\begin{split}
		\Delta^{\theta, \mathbf{X},\mathbf{Y}}_t&=X^{(1)}-Y^{(1)}-k_\theta \int_{0}^{t}\Delta_s^{\theta, \mathbf{X},\mathbf{Y}}ds+\alpha \beta \int_{0}^{t}\sqrt{v_s^{\mathbf{X}}}(\sqrt{\theta_s^{\mathbf{X}}}-\sqrt{\theta_s^{\mathbf{Y}}})dW^{(2)}_s,
	\end{split}
	\end{align}
	Once we notice that
	\begin{align*}
		\E[v^{\mathbf{X}}_s(\sqrt{\theta_s^{\mathbf{X}}}-\sqrt{\theta_s^{\mathbf{Y}}})^2] 
		\le \frac{F_2^{v,X^{(2)}}(s)}{2}+F_{2}^{\theta,\mathbf{X}}(s)+F_{2}^{\theta,\mathbf{Y}}(s),
	\end{align*}
	Lemma \ref{lem:Fptheta} and Corollary \ref{cor:upboundFpv} guarantee that $\left\{\sqrt{v^{\mathbf{X}}_s}(\sqrt{\theta_s^{\mathbf{X}}}-\sqrt{\theta_s^{\mathbf{Y}}}) \right\}_{s \ge 0} \in M^2_t(\Omega)$.
	Taking the conditional expectation in \eqref{eq:firstmomdtheta1}, we have
	\begin{equation}
			\E\left[\Delta^{\theta, \mathbf{X},\mathbf{Y}}_t \mid X^{(1)},Y^{(1)} \right]=X^{(1)}-Y^{(1)}-k_\theta \int_{0}^{t}\E[\Delta_s^{\theta, \mathbf{X},\mathbf{Y}} \mid X^{(1)},Y^{(1)}]ds,
	\end{equation}
	that clearly implies
	\begin{equation*}
	\E\left[\Delta^{\theta, \mathbf{X},\mathbf{Y}}_t \mid X^{(1)}, Y^{(1)}\right]=\left(X^{(1)}-Y^{(1)}\right)e^{-k_\theta t}.
	\end{equation*}
	If we define
	\begin{equation}\label{eq:tau0theta}
		\tau_0^\theta:=\inf\{t \ge t_0: \ \Delta_t^{\theta, \mathbf{X},\mathbf{Y}}=0\},
	\end{equation}
	since for all $t \ge 0$ we have $v_t^{\bm{X}}=v_t^{\bm{Y}}$, by the strong Markov property guaranteed by Proposition~\ref{prop:RFeller} we know that $\Delta^{\theta,\bm{X},\bm{Y}}_t=0$ for all $t \ge \tau_0^{\theta}$. Hence we can prove \eqref{eq:condexptheta} arguing exactly in the same way as in Lemma~\ref{lem:boundcondexp}.
%
%
\end{proof}
Now we can prove the $p$-weak ergodicity of $\bm{\theta}$ under the assumption that \linebreak ${\rm Law}(X^{(2)})=\mathbf{m}_v$. First, notice that once we prove the $p$-weak ergodicity for $p=1$, the case $p>1$ will follow by Theorems \ref{lem:Fpthetaglob} and \ref{thm:equivalentWass}. The assumption ${\rm Law}(X^{(2)})=\mathbf{m}_v$ will be necessary in this first stage since we want to use Lemma \ref{lem:boundcondexptheta}. Indeed, Lemma \ref{lem:boundcondexptheta} provides a form of $\cW_1$-stability of $\theta_t^{\mathbf{X}}$ with respect to $\mathbf{X}$ only when we consider initial data with the same second coordinates. During the proof of the $1$-weak ergodicity, we will use the pathwise uniqueness of the solutions (and then the uniqueness in law of the weak solutions) to prove that $\{\pi_t^{\theta,\mathbf{X}}\}_{t \ge 0}$, where $\pi_t^{\theta,\mathbf{X}}={\rm Law}(\bm{\theta}_t^{\mathbf{X}})$, is Cauchy once restricted on a sequence of times $t_n \uparrow \infty$. Then, as in the proof of Dobrushin's theorem, we need to use the fact that the $1$-Wasserstein distance is \textit{contractive} in a certain sense. To do this, one has to provide a suitable coupling between $\pi_t^{\theta,\mathbf{X}}$ and $\pi_s^{\theta,\mathbf{X}}$ for $t>s$ and then use Lemma \ref{lem:boundcondexptheta}. In order to use the Lemma, we can use the pathwise uniqueness to construct another solution $\{\bm{\theta}_\tau^{\mathbf{Y}}\}_{\tau \ge 0}$ that behaves like $\{\bm{\theta}_\tau^{\mathbf{X}}\}_{\tau \ge 0}$ but starting from $t-s$, so that $\bm{\theta}_s^{\mathbf{Y}}$ will have law $\pi_t^{\theta,\mathbf{X}}$ and then the $1$-Wasserstein distance is controlled by $\E\left[\left|\Delta_t^{\theta,\mathbf{X},\mathbf{Y}}\right|\right]$. Clearly, to carry on the construction, the second coordinate of $\mathbf{Y}$ must have the same law as $v_{t-s}^{\mathbf{X}}$. However, to use Lemma \ref{lem:boundcondexptheta}, we need the second coordinate of $\mathbf{Y}$ to be equal to $X^{(2)}$. Hence, we need to ask at least that ${\rm Law}(v_{t-s}^{\mathbf{X}})={\rm Law}(X^{(2)})$: this is clear if we set ${\rm Law}(X^{(2)})=\mathbf{m}_v$. Once this is done, thanks to the fact that we are working on a \textit{representation space} that satisfies the properties in Proposition \ref{prop:repspace}, we can construct $\mathbf{Y}$ in such a way that $Y^{(2)}=X^{(2)}$ almost surely.

Let us now formalize these arguments.
\begin{thm}\label{thm:weakergo1theta}
	Let $X^{(1)} \in L^{p^\prime}(\Omega:\R_{+,0})$ with $p^\prime>9$ and $X^{(2)}$ with ${\rm Law}(X^{(2)})=\mathbf{m}_v$ be $\cF_0$-measurable. Set also $\mathbf{X}=(X^{(1)},X^{(2)})$. Then the process $\{\bm{\theta}_t^{\mathbf{X}}\}_{t \ge 0}$ is weakly $p$-ergodic for $1 \le p < \frac{p'+3}{4}$. In particular, the limit distribution $\bm{\pi}^{v,\theta}$ is independent of $X^{(1)}$ and belongs to $\cP_p(\R^2)$ for all $p \ge 1$.
\end{thm}
\begin{proof}
	Fix $\delta>0$, define $\nu=e^{-\delta}$, consider the sequence $t_n=\delta n$ and set $\pi_n^{\theta,\mathbf{X}}={\rm Law}(\bm{\theta}_{t_n}^{\mathbf{X}})$ for any $n \ge 0$ and $\mathbf{X} \in L^2(\Omega;\R_{+,0}^2)$. Before proceeding, let us prove that the limit distribution $\bm{\pi}^{v,\theta}$ is independent of $X^{(1)}$, provided we have $1$-weak ergodicity. Indeed, if we let $\bm{X},\bm{Y}$ be as in the statement of the theorem, then $(\bm{\theta}_{t_n}^{\mathbf{X}},\bm{\theta}_{t_n}^{\mathbf{Y}}) \in C(\pi_n^{\theta,\mathbf{X}},\pi_n^{\theta,\mathbf{Y}})$ and
	\begin{equation}\label{eq:W1bivariate}
		\cW_1(\pi_n^{\theta,\mathbf{X}},\pi_n^{\theta,\mathbf{Y}}) \le \E\left[\left|\Delta_{t_n}^{\theta,\mathbf{X},\mathbf{Y}}\right|\right] \le \nu^n \E[\left|X^{(1)}-Y^{(1)}\right|]
	\end{equation}
	by Lemma~\ref{lem:boundcondexptheta}. In particular, since $\nu<1$, this is enough to guarantee that
	\begin{equation*}
		\lim_{n \to +\infty}\cW_1(\pi_n^{\theta,\mathbf{X}},\pi_n^{\theta,\mathbf{Y}})=0,
	\end{equation*}
	hence if $\pi_n^{\theta,\bm{X}}$ converges towards $\bm{\pi}^{v,\theta}$, so does $\pi_n^{\theta,\bm{Y}}$.
	
	Now fix $\bm{X}$ as in the statement of the theorem and let us prove that $\{\bm{\theta}^{\bm{X}}_t\}_{t \ge 0}$ is $1$-weakly ergodic. To do this, we want to prove that $\{\pi_{n}^{\theta,\mathbf{X}}\}_{n \ge 0}$ is a Cauchy sequence. Fix $n>m \ge 0$. We need to find a suitable coupling in $C(\pi_{n}^{\theta,\mathbf{X}},\pi_m^{\theta,\mathbf{X}})$. To do this, consider the random variable $\widetilde{\bm{X}}=(\theta^{\bm{X}}_{t_n-n_m},v^{\bm{X}}_{t_n-t_m})$, with law $\pi^{\theta,\bm{X}}_{n-m}$. By pathwise uniqueness of strong solutions of \eqref{eq:SDEaux2a}, it is clear that $\bm{\theta}^{\widetilde{\bm{X}}}_{t_m}=\bm{\theta}^{\bm{X}}_{t_n}$ almost surely. Furthermore, since ${\rm Law}(X^{(2)})=\bm{m}_v$, then $\{v_t^{\bm{X}}\}$ is stationary and ${\rm Law}(v_{t_m}^{\bm{X}})=\bm{m}_v={\rm Law}(X^{(2)})$. By Item (III) of Proposition \ref{prop:repspace}, we can find a random variable $Y^{(1)} \in \cP(\R)$ such that setting $\bm{Y}=(Y^{(1)},X^{(2)})$ we have ${\rm Law}(\bm{Y})=\pi^{\theta,\bm{X}}_{n-m}={\rm Law}(\widetilde{\bm{X}})$. Hence, since pathwise uniqueness implies uniqueness in law (see \cite[Proposition 1]{yamada1971uniqueness}), we have ${\rm Law}(\bm{\theta}_{t_m}^{\bm{Y}})={\rm Law}(\bm{\theta}_{t_m}^{\widetilde{\bm{X}}})={\rm Law}(\bm{\theta}_{t_n}^{\bm{X}})=\pi_n^{\theta,\bm{X}}$. Hence $(\bm{\theta}_{t_m}^{\bm{Y}},\bm{\theta}_{t_m}^{\bm{X}}) \in C(\pi_n^{\theta,\bm{X}},\pi_m^{\theta,\bm{X}})$. Now we would like to use Lemma \ref{lem:boundcondexptheta}. Clearly, $\bm{X}$ satisfies all the assumptions of the lemma. We only need to verify them on $\bm{Y}$. First, observe that $X^{(2)}$ admits moments of any order. Next, we need to provide a bound for $\E[|Y^{(1)}|^q]$ for some $q>3$. To do this, observe that since $p'>9$, there exist $q>3$ and $\overline{q}>2q-1$ such that $p'>2\overline{q}-1>q$ (for instance $q=\frac{9+p'}{8}$ and $\overline{q}=\frac{3p'+7}{8}$). Then $\bm{X}$ satisfies the assumptions of Theorem \ref{lem:Fpthetaglob} and, in particular, by uniqueness in law,
	\begin{equation*}
		\E[|Y^{(1)}|^q]=F_q^{\theta,\bm{X}}(t_n-t_m) \le \widetilde{C}_q^\theta(\E[|X^{(1)}|^{q}],M_{q}(\bm{m}_v))<\infty.
	\end{equation*}
	Hence $Y^{(1)} \in L^q(\Omega;\R_{+,0})$ for some $q>3$ and then $\bm{Y}$ satisfies the assumptions of Lemma~\ref{lem:boundcondexptheta}. Hence we finally get
	\begin{multline}\label{eq:W1bivariate2}
		\cW_1(\pi_n^{\theta,\mathbf{X}},\pi_m^{\theta,\mathbf{X}}) \le \E\left[\left|\Delta_{t_m}^{\theta,\mathbf{X},\mathbf{Y}}\right|\right]\\ \le \nu^m \E[\left|X^{(1)}-Y^{(1)}\right|] \le 2^{q-1}\nu^m(\E[|X^{(1)}|^q]+\widetilde{C}_q^\theta(\E[|X^{(1)}|^q],M_q(\bm{m}_v))).
	\end{multline}
	Since $\nu<1$, this proves that $\{\pi_n^{\theta,\mathbf{X}}\}_{n \ge 1}$ is a Cauchy sequence in $(\cP_1(\R_{+,0}^2),\cW_1)$, which is a complete metric space. Hence there exists $\bm{\pi}^{v,\theta}$ such that $\pi_n^{\theta,\mathbf{X}} \overset{\cP_1}{\to} \bm{\pi}^{v,\theta}$. Such an argument holds for any $\delta>0$.

	It remains to show that $\bm{\pi}^{v,\theta}$ is independent of $\delta>0$. To do this, let $0<\delta<\delta'$ and set $\widetilde{\pi}_n^{\theta,\mathbf{X}}:={\rm Law}(\bm{\theta}^{\mathbf{X}}_{n\delta'})$ and $\widetilde{\nu}=e^{-\delta'}$. Arguing as before, we can show that there exists a random variable $Y^{(1)}$ such that ${\rm Law}(Y^{(1)},X^{(2)})={\rm Law}(\bm{\theta}^{\mathbf{X}}_{n(\delta-\delta')})$ and ${\rm Law}(\bm{\theta}^{\mathbf{Y}}_{n\delta'})={\rm Law}(\bm{\theta}^{\mathbf{X}}_{n\delta})=\pi_n^{\theta,\mathbf{X}}$, so that $(\bm{\theta}^{\mathbf{Y}}_{n\delta'}, \bm{\theta}^{\mathbf{X}}_{n\delta'}) \in C(\pi_n^{\theta,\mathbf{X}},\widetilde{\pi}_n^{\theta,\mathbf{X}})$ and
	\begin{equation*}
		\cW_1(\pi_n^{\theta,\mathbf{X}},\widetilde{\pi}_n^{\theta,\mathbf{X}}) \le \E\left[\left|\Delta_{n\delta'}^{\theta,\mathbf{X},\mathbf{Y}}\right|\right] \le 2^{q-1}\widetilde{\nu}^n (\E[|X^{(1)}|^q]+\widetilde{C}_q^\theta(\E[|X^{(1)}|^q],M_q(\bm{m}-v))),
	\end{equation*}
	hence also $\widetilde{\pi}_n^{\theta,\bm{X}}$ converges towards $\bm{\pi}^{v,\theta}$ since $\widetilde{\nu}<1$. This finally proves that $\{\bm{\theta}_t^{\bm{X}}\}_{t \ge 0}$ is $1$-weakly ergodic.
	
	Next, we prove that $\{\bm{\theta}_t^{\bm{X}}\}_{t \ge 0}$ is $p$-weakly ergodic for any $1<p<\frac{p'+3}{4}$. Indeed, we can find $\varepsilon>0$ sufficiently small so that $p'>4(p+\varepsilon)-3$. Setting $\overline{p}=\frac{4(p+\varepsilon)-1+p'}{4}$, we know that $\bm{X}$ satisfies the assumptions of Theorem \ref{lem:Fpthetaglob} with exponent $p+\varepsilon$ and then
	\begin{equation*}
		\sup_{t \ge 0}F_{p+\varepsilon}^{\theta,\bm{X}}(t) \le \widetilde{C}_{p+\varepsilon}^{\theta}(\E[|X^{(1)}|^{p+\varepsilon}],M_{p+\varepsilon}(\bm{m}_v)), \quad \mbox{ and } \quad \sup_{t \ge 0}F_{p+\varepsilon}^{v,X^{(2)}}(t)=M_{p+\varepsilon}(\bm{m}_v),
	\end{equation*}
	where the second equality follows by the fact that $\{v_t^{\bm{X}}\}_{t \ge 0}$ is stationary. By Remark \ref{rmk:suffcond} we know that $\{{\rm Law}(\bm{\theta}^{\bm{X}}_t)\}_{t \ge 0}$ has uniformly integrable $p$-moments and then by Theorem \ref{thm:equivalentWass} this implies that ${\rm Law}(\bm{\theta}^{\bm{X}}_t) \overset{\cP_p}{\to} \bm{\pi}^{v,\theta}$.
	
	Finally, notice that since ${\rm Law}(\bm{\theta}^{\bm{X}}_t) \overset{\cP_p}{\to} \bm{\pi}^{v,\theta}$ then $\bm{\pi}^{v,\theta} \in \cP_p(\R_{+,0}^2)$ for any $1 \le p \le \frac{p'+3}{4}$. If we consider $X^{(1)}$ to be a degenerate random variable (i.e., a constant), then it belongs to $L^{p'}(\Omega;\R_{+,0})$ for all $p'>9$, hence $\bm{\pi}^{v,\theta} \in \cP_p(\R_{+,0}^2)$ for any $p \ge 1$.
\end{proof}
In the next steps we will extend the $p$-weak ergodicity result to a more general initial data. To do this, however, we first need to better understand some properties of the limit distribution $\bm{\pi}^{v,\theta}$.

\subsection{Step (III): stationarity and uniqueness of the limit distribution}\label{subs:S3}

In this section, we will make extensive use of the theory of Kolmogorov equations for measures. Before proceeding, however, we prove the following simple lemma.
\begin{lem}\label{lem:asymplem}
	Let $F \in C^1(\R_{+})$ and $\ell,\ell^\prime \in \R$ such that $\lim_{t \to +\infty}F(t)=\ell$ and $\lim_{t \to +\infty}F'(t)=\ell'$. Then $\ell'=0$.
\end{lem}
\begin{proof}
	Let us assume, by contradiction, that $\ell'\not = 0$. We assume first that $\ell'>0$. Then there exists $T>0$ such that for any $t \ge T$ we have $F'(t)>\frac{\ell'}{2}$. Furthermore, by Lagrange's theorem, we have for some $\xi(T,t) \in [T,t]$
	\begin{equation*}
		F(t)=F(T)+F'(\xi(T,t))(t-T)>F(T)+\frac{\ell'}{2}(t-T).
	\end{equation*} 
	Taking the limit as $t \to +\infty$, we have that $\ell=+\infty$, which is absurd since $\ell \in \R$. A similar argument holds for $\ell'<0$.
\end{proof}
Now we can prove the following uniqueness result.
\begin{prop}\label{prop:uniqueness}
	$\bm{\pi}^{v,\theta}$ is the unique invariant measure of the process $\{\bm{\theta}_t\}_{t \ge 0}$ whose first marginal admits finite moment of order $p>9$. Furthermore, if for some $\mathbf{X} \in L^2(\Omega;\R_{+,0})$ there exists a distribution $\overline{\pi}^{v,\theta} \in \cP_p(\R_{+,0}^2)$ for some $p>9$ such that ${\rm Law}(\bm{\theta}_t^{\mathbf{X}}) \overset{\cP}{\to} \overline{\pi}^{v,\theta}$, then $\overline{\pi}^{v,\theta}=\bm{\pi}^{v,\theta}$.
\end{prop}
\begin{proof}
	Denote $\mathbf{W}^{\theta}_t=(W^{(2)}_t,W^{(3)}_t)^{\mathbf{\rm t}}$ for any $t \ge 0$ and consider a function $\varphi \in C^\infty_c(\R_{+}^2)$. By It\^o's formula, we have
	\begin{equation*}
		\varphi(\bm{\theta}^{\mathbf{X}}_t)=\varphi(\mathbf{X})+\int_0^t \mathcal{G}^\theta \varphi(\bm{\theta}^{\mathbf{X}}_s)ds+\int_0^t  \left(\nabla \varphi(\bm{\theta}^{\mathbf{X}}_s)\right)^{\mathbf{\rm t}}\bm{\sigma}_\theta(\bm{\theta}^{\mathbf{X}}_s)d\mathbf{W}^\theta_t,
	\end{equation*}
	where $\mathcal{G}^{\theta}$ is the generator of $\{\bm{\theta}_t\}_{t \ge 0}$ in \eqref{eq:gen2}. Taking the expectation we have
	\begin{equation}\label{eq:afterIto}
	 	\E\left[\varphi(\bm{\theta}^{\mathbf{X}}_t)\right]=\E\left[\varphi(\mathbf{X})\right]+\int_0^t \E\left[\mathcal{G}^\theta \varphi(\bm{\theta}^{\mathbf{X}}_s)\right]ds.
	\end{equation}
	If we set $\pi^{\theta,\mathbf{X}}_t={\rm Law}(\bm{\theta}^{\mathbf{X}}_t)$ for any $t \ge 0$, we can rewrite the previous relation as
	\begin{multline}\label{eq:Kolmmeasure}
		\int_{\R_{+,0}^2}\varphi(\bm{\theta})d\pi^{\theta,\mathbf{X}}_t(\bm{\theta})-\int_{\R_{+,0}^2}\varphi(\bm{\theta})d\pi^{\theta,\mathbf{X}}_0(\bm{\theta})=\int_0^t \mathcal{G}^\theta \varphi(\bm{\theta})d\pi^{\theta,\mathbf{X}}_s(\bm{\theta})ds, \\ \forall \varphi \in C_c^\infty(\R_{+}^2).			
	\end{multline}
	Uniqueness of the solutions $\{\pi_t^{\theta,\mathbf{X}}\}_{t \ge 0}$ of the problem \eqref{eq:Kolmmeasure} with a given initial data $\pi_0^{\theta,\mathbf{X}}$ has been shown in \cite[Theorem 1.2]{manca2008kolmogorov}, as a consequence of the Feller property. Furthermore, if we define $F_\varphi^{\theta,\mathbf{X}}(t)=\E[\varphi(\bm{\theta}_t^{\mathbf{X}})]$ and $H_\varphi^{\theta,\mathbf{X}}(t)=\E[\mathcal{G}^{\theta}\varphi(\bm{\theta}_t^{\mathbf{X}})]$ for $t \ge 0$, we can rewrite \eqref{eq:afterIto} as
	\begin{equation}\label{eq:afterIto2}
		F_\varphi^{\theta,\mathbf{X}}(t)=F_\varphi^{\theta,\mathbf{X}}(0)+\int_0^t H_\varphi^{\theta,\mathbf{X}}(s)ds.
	\end{equation}
	In particular, since $\mathcal{G}^\theta\varphi$ is bounded and the trajectories of $\bm{\theta}^{\mathbf{X}}$ are continuous, $H_\varphi^{\theta,\mathbf{X}}$ is a continuous function and thus $F_\varphi^{\theta,\mathbf{X}} \in C^1(\R_+)$. Taking the derivative on both sides of \eqref{eq:afterIto2} we get
	\begin{equation}\label{eq:afterIto3}
		\frac{d\, F_\varphi^{\theta,\mathbf{X}}(t)}{dt}= H_\varphi^{\theta,\mathbf{X}}(t).
	\end{equation}
	Now consider $\mathbf{X}$ satisfying the assumptions of Theorem \ref{thm:weakergo1theta}. Then, by weak convergence of $\pi_t^{\theta,\mathbf{X}}$ towards $\bm{\pi}^{v,\theta}$, we have
	\begin{equation*}
		\lim_{t \to +\infty}F_\varphi^{\theta,\mathbf{X}}(t)=\int_{\R_{+,0}^2}\varphi(\bm{\theta})d\bm{\pi}^{v,\theta}(\bm{\theta}) \qquad \lim_{t \to +\infty}H_\varphi^{\theta,\mathbf{X}}(t)=\int_{\R_{+,0}^2}\mathcal{G}\varphi(\bm{\theta})d\bm{\pi}^{v,\theta}(\bm{\theta}).
	\end{equation*}
	However, this also means, by \eqref{eq:afterIto3} an Lemma \ref{lem:asymplem}, that
	\begin{equation}\label{eq:0term}
		0=\lim_{t \to +\infty}\frac{d\, F_\varphi^{\theta,\mathbf{X}}(t)}{dt}=\int_{\R_{+,0}^2}\mathcal{G}\varphi(\bm{\theta})d\bm{\pi}^{v,\theta}(\bm{\theta})
	\end{equation}
	Now let us consider $\mathbf{X}$ such that ${\rm Law}(\mathbf{X})=\bm{\pi}^{\theta,v}$. Then, we know that there exists a unique solution to \eqref{eq:Kolmmeasure} with $\pi_0^{\theta,\mathbf{X}}=\bm{\pi}^{\theta,v}$. Such a solution is clearly given by the flow of measures defined as $\pi_t^{\theta,\mathbf{X}}=\bm{\pi}^{\theta,v}$. This proves that $\bm{\pi}^{\theta,v}$ is an invariant measure.
	
	Now assume that for $\mathbf{X} \in L^2(\Omega;\R_{+,0}^2)$ there exists a distribution $\overline{\bm{\pi}}^{v,\theta}$ such that $\pi_t^{\theta,\mathbf{X}}\overset{\cP}{\to} \overline{\bm{\pi}}^{v,\theta}$. Then, arguing as we did for $\bm{\pi}^{\theta,v}$, we know that  
	\begin{equation}\label{eq:0term2}
		\int_{\R_{+,0}^2}\mathcal{G}\varphi(\bm{\theta})d\overline{\bm{\pi}}^{v,\theta}(\bm{\theta})=0, \ \forall \varphi \in C_c^\infty(\R_{+,0})
	\end{equation}
	and then the constant flow $\overline{\bm{\pi}}^{v,\theta}$ is the unique solution of \eqref{eq:Kolmmeasure} with initial data $\overline{\bm{\pi}}^{v,\theta}$. Now let $\mathbf{X}$ with ${\rm Law}(\mathbf{X})=\overline{\bm{\pi}}^{v,\theta}$ and recall that $\pi_t^{\theta,\mathbf{X}}=\overline{\bm{\pi}}^{v,\theta}$ for any $t \ge 0$. The second marginal of $\overline{\bm{\pi}}^{v,\theta}$ is invariant and then ${\rm Law}(X^{(2)})=\mathbf{m}_v$. Moreover, since $\overline{\bm{\pi}}^{v,\theta} \in \cP_p(\R_{+,0}^2)$ for $p>9$, then $X^{(1)} \in L^p(\Omega;\R_{+,0})$ for $p>9$. Hence, by Theorem \ref{thm:weakergo1theta} $\overline{\pi}^{v,\theta}=\pi_t^{\theta,\mathbf{X}} \overset{\cP_1}{\to} \bm{\pi}^{v,\theta}$ and thus $\overline{\pi}^{v,\theta}=\bm{\pi}^{v,\theta}$. The latter property also proves that $\bm{\pi}^{v,\theta}$ is the unique invariant measure whose first marginal admits finite moment of order $p>9$.
\end{proof}
To prove the $p$-weak ergodicity of the three-factor process $\{\bm{R}_t\}_{t \ge 0}$, Theorem \ref{prop:uniqueness} is actually sufficient. However, for completeness, let us prove the $p$-weak ergodicity of the two-factor process in the general case.

\subsection{Step (IV): $p$-weak ergodicity of the two-factor process}\label{subs:S4}
Now we are ready to prove that, under some suitable assumptions on $\bm{X}$, $\{\bm{\theta}_t^{\bm{X}}\}_{t \ge 0}$ is $p$-weakly ergodic.

\begin{lem}\label{thm:pweakergotheta}
	Let $p>9$, $\overline{p}>2p-1$, $p'>2\overline{p}-1$ and $\widetilde{p}=\max\left\{2\overline{p}-1,\frac{p'}{p'-2\overline{p}+1}\right\}$. Consider two $\cF_0$-measurable random variables $X^{(1)}\in L^{p'}(\Omega;\R_{+,0})$ and $X^{(2)} \in L^{\widetilde{p}}(\Omega;\R_{+,0})$. Then the process $\{\bm{\theta}_t^{\mathbf{X}}\}_{t \ge 0}$ is $p$-weakly ergodic with limit distribution $\bm{\pi}^{v,\theta}$.
\end{lem}
\begin{proof}
	Consider any sequence $t_n \to +\infty$ and let $\pi_n^{\theta,\mathbf{X}}={\rm Law}(\bm{\theta}_{t_n}^{\mathbf{X}})$. Recall that, by definition $\mathcal{W}_p(\pi_n^{\theta,\mathbf{X}},\bm{\pi}^{v,\theta}) \ge 0$ and let $\{t_{n_k}\}_{k \ge 0}$ be the subsequence such that $\lim_{k \to +\infty}\mathcal{W}_p(\pi_{n_k}^{\theta,\mathbf{X}},\bm{\pi}^{v,\theta})=\limsup_{n \to +\infty}\mathcal{W}_p(\pi_n^{\theta,\mathbf{X}},\bm{\pi}^{v,\theta})$. Set $\mathcal{M}=\{\pi_{n_k}^{\theta,\mathbf{X}}\}_{k \ge 0}$.

	We want first to prove that $\mathcal{M}$ is relatively compact in $\cP_p(\R_{0,+}^2)$. First, observe that there exists $\varepsilon>0$ such that $\overline{p}>2(p+\varepsilon)-1$. Hence, we are under the hypotheses of Corollary \ref{cor:upboundFpv} and Lemma \ref{lem:Fptheta} and we have
	\begin{multline}\label{eq:controlmoment}
		M_{p+\varepsilon}(\pi_{n_k}^{\theta,\mathbf{X}}) \le 2^{p+\varepsilon-1}(F_{p+\varepsilon}^{\theta,\bm{X}}(t_{n_k})+F_{p+\varepsilon}^{v,X^{(2)}}(t_{n_k})) \\ \le 2^{p+\varepsilon-1}(\widetilde{C}_{p+\varepsilon}^{\theta}(\E[|X^{(1)}|^{p+\varepsilon}],\E[|X^{(2)}|^{p+\varepsilon}])+\widetilde{C}^v_{p+\varepsilon}(\E[|X^{(2)}|^{p+\varepsilon}])).
	\end{multline}
	This inequality already shows, thanks to Remark \ref{rmk:boundcomp}, that $\mathcal{M}$ is relatively compact in $\cP_p(\R_{0,+}^2)$, hence there exists a subsequence, that we still denote by $\pi_{n_k}^{\theta,\mathbf{X}}$, such that 
	\begin{equation*}
		\pi_{n_k}^{\theta,\mathbf{X}} \overset{\cP_p}{\to} \overline{\bm{\pi}}^{v,\theta}.
	\end{equation*}
	Since $\overline{\bm{\pi}}^{v,\theta} \in \cP_{p}(\R_{0,+}^2)$ for some $p>9$, by Proposition \ref{prop:uniqueness} we know that $\overline{\bm{\pi}}^{v,\theta}=\bm{\pi}^{v,\theta}$. Hence \begin{equation*}
		0 \le \liminf_{n \to +\infty}\mathcal{W}_p(\pi_n^{\theta,\mathbf{X}},\bm{\pi}^{v,\theta}) \le \limsup_{n \to +\infty}\mathcal{W}_p(\pi_n^{\theta,\mathbf{X}},\bm{\pi}^{v,\theta})=\lim_{k \to +\infty}\mathcal{W}_p(\pi_{n_k}^{\theta,\mathbf{X}},\bm{\pi}^{v,\theta})=0.
	\end{equation*}
	and, since $t_n$ is arbitrary, we end the proof.

\end{proof}

Let us stress that the assumptions of Lemma~\ref{thm:pweakergotheta} are clearly satisfied by deterministic initial data. We can use this observation to prove the convergence in distribution of $\{\bm{\theta}_t^{\bm{X}}\}$ towards $\bm{\pi}^{v,\theta}$ for more general initial data.
\begin{thm}\label{thm:distrconv}
	For any $\cF_0$-measurable initial data $\mathbf{X} \in L^2(\Omega;\R_{+,0}^2)$  we have ${\rm Law}(\bm{\theta}_t^{\mathbf{X}}) \overset{\cP}{\to} \bm{\pi}^{v,\theta}$. 
\end{thm}
\begin{proof}
	Consider any $\cF_0$-measurable random variable $\mathbf{X} \in L^2(\Omega;\R_{+,0}^2)$ and a bounded continuous function $\varphi:\R_{+,0}^2 \to \R$. 
	Let also $\mathbf{X}_\infty$ be such that ${\rm Law}(\mathbf{X}_\infty)=\bm{\pi}^{\theta,v}$. 
	For any $\mathbf{x} \in \R_{+,0}^2$ consider $\pi_t^{\theta,\mathbf{x}}={\rm Law}(\bm{\theta}_t^{\mathbf{x}})$ and define $F_\varphi^{\bm{\theta}}(t;\mathbf{x})=\E[\varphi(\bm{\theta}_t^{\bm{x}})]$. By Lemma \ref{thm:pweakergotheta} we know that 
	\begin{equation*}
		\lim_{t \to +\infty}F_\varphi^{\bm{\theta}}(t;\mathbf{x})=\E[\varphi(\bm{X}_\infty)].
	\end{equation*}
	Recalling that since $\bm{X}$ is $\cF_0$-measurable it is independent of the Brownian motion $\{(W^{(2)}_t,W^{(3)}_t)\}_{t \ge 0}$, we have
	\begin{equation*}
		\lim_{t \to +\infty}\E\left[\varphi(\bm{\theta}_t^{\mathbf{X}})\right]=\lim_{t \to +\infty}\E\left[\mathfrak{F}_\varphi^{\theta}(t;\mathbf{X})\right]=\E[\varphi(\bm{X}_\infty)],
	\end{equation*}
	where we could take the limit inside the expectation by a simple application of the dominated convergence theorem. Since $\varphi \in C_{b}(\R_{+,0}^2)$ is arbitrary, we get the statement.
\end{proof}
\begin{remark}
	As a direct consequence of Theorem \ref{thm:distrconv} we get that $\bm{\pi}^{v,\theta}$ is the unique solution of
	\begin{equation*}
		\int_0^t \mathcal{G}^{\theta}\varphi(\bm{\theta})d\bm{\pi}^{\theta,v}ds=0, \ \forall \varphi \in C^\infty_c(\R_{+,0}^2)
	\end{equation*}
	in $\cP_2(\R_{+,0}^2)$.
\end{remark}
We are finally ready to prove the $p$-weak ergodicity of $\{\bm{\theta}_t^{\bm{X}}\}_{t \ge 0}$ for a larger class of initial data $\bm{X}$.
\begin{thm}
	Let $p \ge 2$, $\overline{p}>2p-1$, $p^\prime>2\overline{p}-1$ and $\widetilde{p}=\max\left\{2\overline{p}-1,\frac{p^\prime}{p^\prime-2\overline{p}+1}\right\}$. Consider two $\cF_0$-measurable random variables $X^{(1)} \in L^{p^\prime}(\Omega;\R_{+,0})$ and $X^{(2)} \in L^{\widetilde{p}}(\Omega;\R_{+,0})$. Then $\{\bm{\theta}_t^{\bm{X}}\}_{t \ge 0}$ is $p$-weakly ergodic with limit distribution $\bm{\pi}^{v,\theta}$.
\end{thm}
\begin{proof}
	Since $\bm{X} \in L^2(\Omega;\R_{+,0}^2)$, Theorem \ref{thm:distrconv} guarantees that ${\rm Law}(\bm{\theta}_t^{\bm{X}}) \overset{\cP}{\to} \bm{\pi}^{v,\theta}$. Furthermore, if we consider $\varepsilon>0$ such that $\overline{p}>2(p+\varepsilon)-1$, we can use Theorem \ref{lem:Fpthetaglob} and Corollary \ref{cor:upboundFpv} to state that $\sup_{t \ge 0}M_{p+\varepsilon}({\rm Law}(\bm{\theta}_t^{\bm{X}}))<\infty$. By Remark \ref{rmk:suffcond}, this guarantees that $\{{\rm Law}(\bm{\theta}_t^{\bm{X}})\}_{t \ge 0}$ has uniformly integrable $p$-moments and then, by Theorem \ref{thm:equivalentWass}, ${\rm Law}(\bm{\theta}_t^{\bm{X}}) \overset{\cP_p}{\to} \bm{\pi}^{v,\theta}$.
\end{proof}

\section{The $p$-weak ergodicity of the $CIR^3$ process}\label{sec:finalsec}

We are finally ready to discuss full three-factor model $\{\mathbf{R}_t\}_{t \ge 0}:=\{(R_t,\theta_t,v_t)\}_{t \ge 0}$. To prove.... we will adopt exactly the same strategy as in  Section \ref{subsec:twofactor}. For this reason, some proofs will be only sketched, underlining, at the same time, the parts in which they differ from the analogous ones in Section \ref{subsec:twofactor}. Again, we need some notation. For any $\cF_0$-measurable random variable $\mathbf{X}=(X^{(1)},X^{(2)},X^{(3)}) \in L^2(\Omega;\R_{+,0}^3)$, we let $\{\bm{R}_t^{\mathbf{X}}\}_{t \ge 0}=\{(R^{\mathbf{X}}_t,\theta^{\mathbf{X}}_t,v^{\mathbf{X}}_t)\}_{t \ge 0}$ be the pathwise unique strong solution of \eqref{Our_Model_true} with initial data $\bm{X}$. For any $p \ge 0$, we set $F_p^{R,\mathbf{X}}(t)=\E[|R_t^{\mathbf{X}}|^p]$, $G_{p}^{R,\mathbf{X}}(t)=\E[|R_t^{\mathbf{X}}|^pv_t^{\mathbf{X}}]$, and $J_p^{R,\mathbf{X}}(t)=\E[|R_t^{\mathbf{X}}|^p\theta_t^{\mathbf{X}}]$. We also define $\Delta_t^{R,\mathbf{X},\mathbf{Y}}=R_t^{\mathbf{X}}-R_t^{\mathbf{Y}}$ for $t \ge 0$ and any $\cF_0$-measurable random variables $\mathbf{X},\bm{Y} \in L^2(\Omega;\R_{+,0}^3)$. 

First, arguing as in Subsection \ref{subs:S1}, we need to prove that under some suitable assumptions on the initial data it holds $\sup_{t \ge 0}F_p^{R}(t)<\infty$. First, we prove a local bound, as in Lemma \ref{lem:Fptheta}. 
\begin{lem}\label{lem:FpR}
	Let $p \ge 2$, $p^\prime>2p-1$ and $\widetilde{p}=\max\left\{2p-1,\frac{p^\prime}{p^\prime-2p+1}\right\}$. Consider the three $\cF_0$-measurable random variables $X^{(1)},X^{(2)} \in  L^{p^\prime}(\Omega;\R_{+,0})$ and $X^{(3)} \in  L^{\widetilde{p}}(\Omega;\R_{+,0})$. Then, for any $T>0$,
	\begin{equation*}
		\sup_{t \in [0,T]}F_p^{R,\mathbf{X}}(t)<\infty.
	\end{equation*}	
\end{lem}
\begin{proof}
	Let $p \ge 2$, and observe that, by It\^o's formula,
	\begin{align}\label{eq:Rp}
		\begin{split}
			(R_{\tau^{R,\bm{X},M}_t}^{\mathbf{X}})^p&=\left|X^{(1)}\right|^p+\int_0^{\tau^{R,\bm{X},M}_t} \left(-k p \left|R^{\mathbf{X}}_s\right|^p+\left(k p \theta_s^{\mathbf{X}}+\frac{\alpha^2p(p-1)}{2}v_s^{\mathbf{X}}\right)|R_s^{\mathbf{X}}|^{p-1}\right)ds\\
			&\qquad +\alpha p\int_0^{\tau^{R,\bm{X},M}_t} \sqrt{v_{\tau^{R,\bm{X},M}_s}^{\mathbf{X}}}|R_{\tau^{R,\bm{X},M}_s}^{\mathbf{X}}|^{p-\frac{1}{2}}dW_s^{(1)}
		\end{split}
	\end{align}
	where, for any $M,t>0$,
	\begin{equation}\label{eq:TRtheta}
		T_M^{R,\mathbf{X}}:=\inf\{t \ge 0: \ R_t^{\mathbf{X}} \vee \theta_t^{\bm{X}} \vee v_t^{\mathbf{X}}   \ge  M\}, \quad  \tau^{R,\bm{X},M}_t=t \wedge T_M^{R,\bm{X}}.
	\end{equation}
	Arguing as in Lemma \ref{lem:Fptheta} we can show that $\{\sqrt{v_{\tau^{R,\bm{X},M}_s}^{\mathbf{X}}}|R_{\tau^{R,\bm{X},M}_s}^{\mathbf{X}}|^{p-\frac{1}{2}}\}_{s \ge 0} \in M^2_t(\Omega)$ for any $t \ge 0$ and then, taking the expectation, we have
	\begin{align}\label{eq:Rp3}
		\begin{split}
			\E\left[(R_{\tau^{R,\bm{X},M}_t}^{\mathbf{X}})^p\right]&=\E\left[\left|X^{(1)}\right|^p\right]\\
			&+\E\left[\int_0^{\tau^{R,\bm{X},M}_t} \left(-k p \left|\theta^{\mathbf{X}}_s\right|^p+\left(k p \theta_s^{\mathbf{X}}+\frac{\alpha^2p(p-1)}{2}v_s^{\mathbf{X}}\right)|R_s^{\mathbf{X}}|^{p-1}\right)ds\right].
		\end{split}
	\end{align}
	Setting
	\begin{equation*}
		\varepsilon=\left(\frac{2 k p}{\alpha^2(p-1)^2+2k(p-1)}\right)^{1-\frac{1}{p}}
	\end{equation*}
	and using Young's inequality to get
	\begin{equation}\label{eq:YoungR2}
		\theta_s^{\mathbf{X}}\left|R_s^{\mathbf{X}}\right|^{p-1} \le \frac{|\theta_s^{\mathbf{X}}|^p}{p\varepsilon^p}+\frac{(p-1)\varepsilon^{\frac{p}{p-1}}}{p}\left|R_s^{\mathbf{X}}\right|^{p}, \qquad v_s^{\mathbf{X}}\left|R_s^{\mathbf{X}}\right|^{p-1} \le \frac{\left|v_s^{\mathbf{X}}\right|^p}{p\varepsilon^p}+\frac{(p-1)\varepsilon^{\frac{p}{p-1}}}{p}\left|R_s^{\mathbf{X}}\right|^{p},
	\end{equation}
 	we achieve from \eqref{eq:Rp3}
	\begin{align}\label{eq:Rp5}
		\begin{split}
			\E&\left[(R_{t \wedge T_M^{\theta,\mathbf{X}}}^{\mathbf{X}})^p\right]=\E\left[\left|X^{(1)}\right|^p\right]\\
			&+\left(\frac{\alpha^2\beta^2(p-1)^2+2k_\theta(p-1)}{2pk_\theta}\right)^{p-1}\E\left[\int_0^{t \wedge T_M^{\theta,\mathbf{X}}} \left(k \left|\theta_s^{\mathbf{X}}\right|^{p} +\frac{\alpha^2(p-1)}{2}\left|v_s^{\mathbf{X}}\right|^p\right)ds\right].
		\end{split}
	\end{align}
	Taking the limit as $M \to +\infty$, using Fatou's lemma, the monotone convergence theorem, Lemma \ref{lem:Fptheta} and Corollary \ref{cor:upboundFpv}, we prove the statement.
\end{proof}
As a direct consequence, we get the following local bounds on $G_p^{R,\bm{X}}$ and $J_p^{R,\bm{X}}$.
\begin{lem}\label{lem:upGpR}
	Let $p \ge 2$, $p^\prime>2\overline{p}-1$ and $\widetilde{p}=\max\left\{2\overline{p}-1,\frac{p^\prime}{p^\prime-2\overline{p}+1}\right\}$. Consider three $\cF_0$-measurable random variables $X^{(1)},X^{(2)} \in L^{p^\prime}(\Omega;\R_{+,0})$ and $X^{(3)} \in  L^{\widetilde{p}}(\Omega;\R_{+,0})$. Then, for any $T>0$
	\begin{equation*}
		\sup_{t \in [0,T]}G_p^{R,\mathbf{X}}(t)<\infty, \qquad \sup_{t \in [0,T]}J_{p-1}^{R,\mathbf{X}}(t)<\infty.
	\end{equation*}	
\end{lem}
\begin{proof}
	By Young's inequality we have
	\begin{equation*}
		G_p^{R,\bm{X}}(t) \le \frac{p}{\overline{p}}F_{\overline{p}}^{R,\bm{X}}(t)+\frac{\overline{p}-p}{\overline{p}}F^{v,X^{(3)}}_{\frac{\overline{p}}{\overline{p}-p}}(t) \ \mbox{ and } \ 
		J_{p-1}^{R,\mathbf{X}}(t) \le \frac{1}{p}F_p^{\theta,(X^{(2)},X^{(3)})}(t)+\frac{p-1}{p}F_p^{R,\mathbf{X}}(t).
	\end{equation*}
	Hence, the statement follows by Lemmas \ref{lem:FpR} and \ref{lem:Fptheta}, together with Corollary \ref{cor:upboundFpv}.	
\end{proof}
It remains to prove a global bound on $F_p^{R,\bm{X}}$ as in Theorem \ref{lem:Fpthetaglob}.
\begin{thm}\label{lem:FpRglob}
	Let $p \ge 2$, $\overline{p}>2p-1$, $p^\prime>2\overline{p}-1$ and $\widetilde{p}=\max\left\{2\overline{p}-1,\frac{p^\prime}{p^\prime-2\overline{p}+1}\right\}$. Consider three $\cF_0$-measurable random variables $X^{(1)},X^{(2)} \in L^{p^\prime}(\Omega;\R_{+,0})$ and $X^{(3)} \in L^{\widetilde{p}}(\Omega;\R_{+,0})$ . Then $F_p^{R,\mathbf{X}}$ and there exists a function $\widetilde{C}_p^{R}:\R_{0,+}^3 \to \R_+$ that is non-decreasing in all its variables and
	\begin{equation}\label{eq:ineqglobalR}
		\sup_{t \ge 0}F_p^{R,\mathbf{X}}(t) \le \widetilde{C}_p^{R}(\E[|X^{(1)}|^p],\E[|X^{(2)}|^p],\E[|X^{(3)}|^p]).
	\end{equation}
\end{thm}
\begin{proof}
	Continuity of $F_p^{R,\bm{X}}$ is shown arguing exactly in the same way as in Theorem \ref{lem:Fpthetaglob}, using the local bounds obtained in Lemmas \ref{lem:FpR} and \ref{lem:upGpR}. Furthermore, the same arguments lead to \eqref{eq:ineqglobalR}, employing \eqref{eq:YoungR2} with $\varepsilon=\left(\frac{k p}{\alpha^2(p-1)^2+2k (p-1)}\right)^{1-\frac{1}{p}}$ and the global bounds provided in Theorem \ref{lem:Fpthetaglob} and Corollary \ref{cor:upboundFpv}.

\end{proof}
Next, to proceed as in Subsection \ref{subs:S2}, we need to provide an estimate on the expected value of $|\Delta_t^{R,\bm{X},\bm{Y}}|$ provided $(X^{(2)},X^{(3)})=(Y^{(2)},Y^{(3)})$. The proof is the same as the one of Lemma \ref{lem:boundcondexptheta} and thus is omitted.
\begin{lem}\label{lem:boundcondexpR}
	Let $p>3$ and $\widetilde{p}=\max\left\{\frac{p^\prime}{p^\prime-3},3\right\}$. Consider four $\cF_0$-measurable random variables $X^{(1)},Y^{(1)},X^{(2)} \in L^{p}(\Omega; \R_+)$ and $X^{(3)}\in L^{\widetilde{p}}(\Omega; \R_+)$. Define $\mathbf{X}=(X^{(1)},X^{(2)},X^{(3)})$ and $\mathbf{Y}=(Y^{(1)},X^{(2)},X^{(3)})$. For any $t \ge 0$ we have
	\begin{equation*}
		\E\left[|\Delta_t^{R, \mathbf{X},\mathbf{Y}}|\right] \le \E\left[\left|X^{(1)}-Y^{(1)}\right|\right]|e^{-k t}.
	\end{equation*}
\end{lem}
Then we can prove the $p$-weak ergodicity of $\mathbf{R}_t^{\mathbf{X}}$ when ${\rm Law}(\mathbf{X}_{2,3})=\bm{\pi}^{v,\theta}$.
\begin{thm}\label{thm:weakergo1R}
	Let $p^\prime > 9$. The process $\{\bm{R}_t^{\mathbf{X}}\}_{t \ge 0}$ is weakly $p$-ergodic for any three $\cF_0$-random variables $X^{(1)} \in L^{p^\prime}(\Omega;\R_{+,0})$ and $X^{(2)},X^{(3)}$ such that ${\rm Law}(X^{(2)},X^{(3)})=\bm{\pi}^{V,\theta}$ and any $1 \le p < \frac{p^\prime+3}{4}$. Furthermore, the limit distribution $\bm{\pi}$ is independent of $X^{(1)}$ and belongs to $\cP_p(\R^3_{+,0})$ for any $p \ge 1$.
%
%
%
%
\end{thm}
\begin{proof}
	Fix $\delta>0$, define $\nu=e^{-\delta}$, $t_n=\delta n$ and $\pi_n^{R,\mathbf{X}}={\rm Law}(\bm{R}_{t_n}^{\mathbf{X}})$, for any $\mathbf{X}$ as in the statement of the theorem and $n \ge 0$. Since the proof is exactly the same as the one of Theorem \ref{thm:weakergo1theta}, let us only discuss, for fixed $n \ge m \ge 1$, the construction of the random variable $\bm{Y}$ such that $(Y^{(2)},Y^{(3)})=(X^{(2)},X^{(3)})$ and $\cW_1(\pi_n^{\theta,\bm{X}},\pi_m^{\theta}) \le \E[|\Delta_{t_m}^{\theta,\bm{X},\bm{Y}}] \le C\nu^m$ for some constant $C>0$ independent of $n,m$. Denote $\bm{R}_t^{\bm{X}}=(R_t^{\bm{X}},\bm{\theta}_t^{(X^{(2)},X^{(3)})})$. We notice that ${\rm Law}(X^{(2)},X^{(3)})=\bm{\pi}^{v,\theta}={\rm Law}(\bm{\theta}^{(X^{(2)},X^{(3)})}_{t_n-t_m})$, since $\bm{\pi}^{v,\theta}$ is the stationary distribution of the two-factor process. By using Item (III) of Proposition \ref{prop:repspace}, we know that there exists a random variable $Y^{(1)}$ such that ${\rm Law}(Y^{(1)})={\rm Law}(R^{\bm{X}}_{t_n-t_m})$ and, setting $\bm{Y}=(Y^{(1)},X^{(2)},X^{(3)})$, ${\rm Law}(\bm{Y})={\rm Law}(\bm{R}_{t_n-t_m}^{\bm{X}})$. By both uniqueness in law and pathwise uniqueness of the strong solution, it is clear that ${\rm Law}(\bm{R}_{t_m}^{\bm{Y}})={\rm Law}(\bm{R}_{t_n}^{\bm{X}})=\pi_n^{R,\bm{X}}$. The couple $(\bm{R}_{t_m}^{\bm{Y}},\bm{R}_{t_m}^{\bm{X}}) \in C(\pi_n^{R,\bm{X}},\pi_m^{R,\bm{X}})$ is the desired coupling.
\end{proof}
Next, we state, as in Subsection \ref{subs:S3}, that $\bm{\pi}$ is the unique invariant (and limit) distribution for the three-factor process. Since the proof is exactly the same as the one of Proposition \ref{prop:uniqueness}, except for the fact that we employ the generator $\mathcal{G}$ of the three-factor process defined in \eqref{eq:gen3}, we omit it.
\begin{prop}\label{prop:uniquenessR}
	$\bm{\pi}$ is the unique invariant measure of the process $\{\bm{R}_t\}_{t \ge 0}$ whose first marginal admits finite moment of order $p>9$. Furthermore, if for some $\mathbf{X} \in L^2(\Omega;\R_{+,0}^3)$ there exists a distribution $\overline{\pi} \in \cP_p(\R_{+,0}^3)$ for some $p>9$ such that ${\rm Law}(\bm{R}_t^{\mathbf{X}}) \overset{\cP}{\to} \overline{\pi}$, then $\overline{\pi}=\bm{\pi}$.
\end{prop}
Next we proceed as in \ref{subs:S4} to finally prove Theorem~\ref{thm:pweakergoR}.
\begin{proof}[Proof of Theorem~\ref{thm:pweakergoR}]
	Let us first prove the statement under the additional assumption that $p>9$. The argument is analogous to the one of Theorem \ref{thm:weakergo1theta}. Indeed, consider any sequence $t_n \to +\infty$ and let $\pi_n^{\mathbf{X}}={\rm Law}(\bm{R}_{t_n}^{\mathbf{X}})$. We extract a subsequence $\{t_{n_k}\}_{k \ge 1}$ such that $\limsup_{n \to +\infty}\mathcal{W}_p(\pi_n^{\mathbf{X}},\bm{\pi})=\lim_{k \to +\infty}\mathcal{W}_p(\pi_{n_k}^{\mathbf{X}},\bm{\pi})$. We need to prove that $\{\pi_{n_k}^{R,\mathbf{X}}\}_{k \ge 0}$ is relatively compact. This is clear once we notice that there exists $\varepsilon>0$ such that $\overline{p}>2(p+\varepsilon)-1$ and then 
	\begin{equation*}
		\sup_{k \ge 1}M_{p+\varepsilon}(\pi_{n_k}^{\bm{X}}) \le \sup_{k \ge 1}3^{p+\varepsilon-1}(F_{p+\varepsilon}^{R,\bm{X}}(t_{n_k})+F_{p+\varepsilon}^{\theta,(X^{(2)},X^{(3)})}(t_{n_k})+F_{p+\varepsilon}^{v,X^{(3)}}(t_{n_k}))<\infty,
	\end{equation*}
	where we used Theorems \ref{lem:FpRglob} and \ref{lem:Fpthetaglob} together with Corollary \ref{cor:upboundFpv}. Hence, by Remark \ref{rmk:boundcomp}, we know that there exists $\overline{\bm{\pi}} \in \cP_p(\R_{+,0}^3)$ such that, up to a subsequence, $\pi_{n_k}^{\bm{X}} \overset{\cP_p}{\to} \overline{\bm{\pi}}$. By Proposition \ref{prop:uniquenessR} we get $\overline{\bm{\pi}}=\bm{\pi}$ and then
	\begin{equation*}
		0 \le \liminf_{n \to +\infty}\mathcal{W}_p(\pi_n^{\mathbf{X}},\bm{\pi})\le \limsup_{n \to +\infty}\mathcal{W}_p(\pi_n^{\mathbf{X}},\bm{\pi})=\lim_{k \to +\infty}\mathcal{W}_p(\pi_{n_k}^{\mathbf{X}},\bm{\pi})=0.
	\end{equation*}
	In particular, this implies that the statement of the theorem holds for deterministic initial data. 
	
	Let now $p \ge 2$ be any exponent satisfying the assumptions of the theorem. Consider any function $\varphi \in C_b(\R_{+,0}^3)$ and for any $\bm{x} \in \R_{+,0}^3$ we define $F^{\bm{R}}_\varphi(t;\bm{x})=\E[\varphi(\bm{R}_t^{\bm{x}})]$. Recall that if we consider a random variable $\bm{X}_\infty$ such that ${\rm Law}(\bm{X}_\infty)=\bm{\pi^{\theta,v}}$ it holds
	\begin{equation*}
		\lim_{t \to +\infty}F^{\bm{R}}_\varphi(t;\bm{x})=\E[\varphi(\bm{X}_\infty)].
	\end{equation*}
	Furthermore, since $\bm{X}$ is $\cF_0$-measurable and then independent of the Brownian motion $\{\bm{W}_t\}_{t \ge 0}$, we have
	\begin{equation*}
		\lim_{t \to +\infty}\E\left[\varphi(\bm{R}_t^{\bm{X}})\right]=\lim_{t \to +\infty}\E\left[F^{\bm{R}}_\varphi(t;\bm{X})\right]=\E[\varphi(\bm{X}_\infty)],
	\end{equation*}
	where we also used the dominated convergence theorem. Hence $\pi_t^{\bm{X}}:={\rm Law}(\bm{R}_t^{\bm{X}}) \overset{\cP}{\to} \bm{\pi}$. Furthermore, there exists $\varepsilon>0$ such that $\overline{p}>2(p+\varepsilon)-1$ and then, by Theorems \ref{lem:Fpthetaglob} and \ref{lem:FpRglob} together with Corollary \ref{cor:upboundFpv}, we get
	\begin{equation*}
		\sup_{t \ge 0}M_{p+\varepsilon}(\pi_{t}^{\bm{X}}) \le \sup_{t \ge 0}3^{p+\varepsilon-1}(F_{p+\varepsilon}^{R,\bm{X}}(t)+F_{p+\varepsilon}^{\theta,(X^{(2)},X^{(3)})}(t)+F_{p+\varepsilon}^{v,X^{(3)}}(t))<\infty,
	\end{equation*}
	hence $\{\pi_t^{\bm{X}}\}_{t \ge 0}$ has uniformly integrable $p$-moments by Remark \ref{rmk:suffcond} and then, by Theorem \ref{thm:equivalentWass}, $\pi_t^{\bm{X}} \overset{\cP_p}{\to}\bm{\pi}$.
\end{proof}

\section{Conclusion} \label{Sec:Conclusion}


In this study, we considered a possible variation on the Cox-Ingersoll-Ross (CIR) process, whose interest in mathematical finance is already established. Various modifications of the original CIR model, including those incorporating fractional Brownian motion, mixed Brownian motion, time-change mechanisms, jumps, and variance gamma processes, have been discussed in literature.

Within this context, in \cite{Ceci2024} a three-factor model known as the $CIR^3$ model has been introduced, where both the trend and volatility are stochastic and correlated. Further properties of such a model, together with some suitable applications, have been investigated in \cite{Ascione2024Jan,Ascione2023Dec}. The main focus of this paper is on exploring the stationary and limit distributions of the model and on establishing its Wasserstein ergodicity, a task that requires nuanced arguments involving topological aspects of Wasserstein spaces and Kolmogorov equations for measures. This strategy can also be applied to prove the Wasserstein ergodicity of the well-known three-factor Chen model.
Specifically, we first explored the $p$-weak ergodicity of the $CIR$ process $v_t$. While the process's strong ergodicity is well-known, existing classical results do not address $p$-weak ergodicity for $p>1$ or the rate of convergence. Hence, we provided detailed insights into the $p$-Wasserstein convergence of the single-factor process.
Furthermore, our investigation focused on establishing the existence of the limit distribution of the two-factor model, in line with the inherent "cascade structure" of the equation. This process is subdivided into four steps, each thoroughly elaborated upon.
Finally, we presented the proof of Theorem \ref{thm:pweakergoR} subsequent to establishing the limit distribution of the two-factor model.

In summary, in conjunction with \cite{Ascione2023Dec,Ascione2024Jan,Ceci2024}, we have thoroughly examined the properties and practical applications of the $CIR^3$ model. Our primary emphasis was on investigating its limit and stationary distribution properties to establish its $p$-weak ergodicity under appropriate initial data assumptions. This property can be used to identify and study stationary time series, as well as the long-term behavior of non-stationary time series. Additionally, we have highlighted the broader mathematical interest in this problem beyond its financial implications, suggesting potential extensions to more complex systems of stochastic differential equations with a "leader-follower" structure in future research efforts.

\section*{Declarations}
No funding was received for conducting this study. The authors have no relevant financial or non-financial interests to disclose.

\bibliographystyle{apalike} 
\bibliography{MyBibCredit}

\appendix
\section{Proof of Proposition \ref{prop:repspace}}\label{app:reprspace}
We provide here the construction of the representation space $(\Omega, \cF, \{\cF_t\}_{t \ge 0}, \bP)$. First, let us consider the canonical space $(\Omega^\prime, \cF^\prime, \{\cF_t^\prime\}, \bP^\prime)$ of a $3$-dimensional standard Brownian motion $\overline{W}=\{(\overline{W}^{(1)}_t,\overline{W}^{(2)}_t,\overline{W}^{(3)}_t)\}_{t \ge 0}$ conditioned to $(\overline{W}^{(1)}_0,\overline{W}^{(2)}_0,\overline{W}^{(3)}_0)$. The definition of a canonical space and a canonical Feller process is given in \cite{revuz2013continuous}, but let us recall here the fundamental idea of the construction. We let $\Omega^\prime=C(\R_{+,0};\R^3)$, i.e. the space of continuous functions from $[0,+\infty)$ to $\R^3$ and for any $\omega \in \Omega^\prime$ we define $\overline{W}_t(\omega)=\omega(t)$, i.e. the evaluation map. Once we equip $\R^3$ with its Borel $\sigma$-algebra, we can consider, for any $t \ge 0$ $\cF^0_t$ as the $\sigma$-algebra on $\Omega^\prime$ generated by $\overline{W}_s$ for $s \le t$ and $\cF^0_\infty$ as the $\sigma$-algebra generated by the union of all $\cF^0_t$. By Kolmogorov's extension theorem, we can define a probability measure $\bP^\prime$ (that is the Wiener measure) on $\cF^0_\infty$ such that $\overline{W}$ is a standard Brownian motion. Once we do this, we define, for any $t \ge 0$, $\cF^\prime_t$ and $\cF^\prime$ as the completion respectively of $\cF^0_t$ and $\cF_\infty^0$ with respect to $\bP^\prime$. 

Next let us consider the probability space $(\Omega'',\cF'',\bP'')$, where $\Omega''=[0,1]$, $\bP''$ is the Lebesgue measure on $([0,1],\cB[0,1])$ and $\cF''$ is the completion of $\cB[0,1]$ with respect to $\bP''$. It is well-known that for any probability distribution $\mu \in \cP(\R)$ there exists a random variable $X: \Omega'' \to \R$ with $\mu={\rm Law}(X)$. 

Finally, we consider the space $(\Omega, \cF, \{\cF_t\}_{t \ge 0}, \bP)$ where $\Omega=\Omega' \times \Omega''$, $\bP=\bP' \otimes \bP''$, $\cF$ is the completion of $\cF' \otimes \cF''$ and $\cF_t$ is the completion of $\cF'_t \otimes \cF''$. In this space, we can define the process
\begin{equation*}
	\widetilde{W}(\omega)=\widetilde{W}(\omega',\omega'')=\overline{W}(\omega'').
\end{equation*}
We observe that this is a Brownian motion on $(\Omega, \cF, \{\cF_t\}_{t \ge 0}, \bP)$. Indeed, $\widetilde{W}:\Omega \to C(\R_{+,0})$ by definition. Furthermore, for any measurable set $A \subset C(\R_{+,0})$, we have
\begin{equation*}
	\bP(\{\omega \in \Omega: \widetilde{W}(\omega) \in A\})=\bP(\Omega' \times \{\omega'' \in \Omega: \overline{W}(\omega'') \in A\})=\bP''(\{\omega'' \in \Omega: \overline{W}(\omega'') \in A\}),
\end{equation*}
i.e. ${\rm Law}(\widetilde{W})={\rm Law}(\overline{W})$. Now consider $\Sigma$ as in \eqref{eq:corrmat} and let
\begin{equation*}
	\mathcal{R}=\begin{pmatrix} \sqrt{1-\rho_\theta^2-\rho_v^2} & \rho_\theta & \rho_v\\
		0 & 1 & 0 \\
		0 & 0 & 1
	\end{pmatrix}.
\end{equation*}
Then, clearly, the process $W=\{(W^{(1)}_t,W^{(2)}_t,W^{(3)}_t)\}_{ t \ge 0}$ defined, for any $t \ge 0$ as $W_t=\mathcal{R}\widetilde{W}_t$ is a correlated Brownian motion with infinitesimal correlation matrix $\Sigma$. Moreover, for any $\mu \in \cP(\R)$, we can consider the cumulative probability distribution function $F_\mu(x)=\mu(-\infty,x]$ and its generalized inverse
\begin{equation*}
	F_\mu^{\leftarrow}(y)=\inf\{x \in \R: \ y \le F_\mu(x)\}, \ y \in [0,1].
\end{equation*}
Then we can define
\begin{equation}\label{eq:Skorokhodconstr}
	X(\omega):=F_\mu^{\leftarrow}(\omega')
\end{equation}
so that, for any real number $x \in \R$ we have
\begin{equation*}
	\bP(\{\omega \in \Omega: X(\omega) \le x\})=\bP(\{\omega' \in \Omega': \ F_\mu^{\leftarrow}(\omega')\}\times \Omega'')=F_\mu(x).
\end{equation*}
To prove Item (III), observe that we can assume, without loss of generality, that $d_2=1$. By the disintegration theorem, we know that there exists a Borel-measurable family $\{\nu(\cdot|x)\}_{x \in \R^{d_1}}$ of probability measures such that for any $A \subset \R^{d_1+1}$, denoting, for any $x \in \R$, $A_x=\{y \in \R: \ (x,y) \in A\}$,
\begin{equation*}
	\pi(A)=\int_{\R}\nu(A_x|x)d\mu(x).
\end{equation*}
For each one of these probability measures, we consider the random variable $Y_x$ constructed via \eqref{eq:Skorokhodconstr} and then the random variable $Y=Y_X$. Then, by construction, ${\rm Law}(X,Y)=\pi$.
\qed

\end{document}